%% file: main.tex
\documentclass[11pt,reqno]{amsart}

\usepackage{xcolor}
\usepackage[T1]{fontenc}
\usepackage{amsmath}							
\usepackage{amssymb}
\usepackage{amsthm}
\usepackage{amscd}
\usepackage{amsfonts}
\usepackage{mathabx}
\usepackage{mathtools}
\usepackage[english]{babel}
\usepackage{stmaryrd}
\usepackage[all]{xy}
\usepackage{array}
\usepackage{multirow}       
\usepackage{comment}
\usepackage{subfiles}       
\usepackage[autostyle]{csquotes}

\usepackage{caption}
\usepackage{subcaption}
\usepackage{euler}

\usepackage{extarrows}

\usepackage[colorlinks, linktocpage, citecolor = purple, linkcolor = blue]{hyperref}
\usepackage{color}
\usepackage{comment}

\usepackage{tikz}									
\usetikzlibrary{matrix}
\usetikzlibrary{patterns}
\usetikzlibrary{positioning}
\usetikzlibrary{decorations.pathmorphing}
\usetikzlibrary{decorations.markings}
\usetikzlibrary{cd}

\usepackage{ytableau}
\usepackage{fullpage}
\usepackage[shortlabels]{enumitem}

\newtheorem{maintheorem}{Theorem}	

\newtheorem{theorem}{Theorem}[section]
\newtheorem{lemma}[theorem]{Lemma}
\newtheorem{proposition}[theorem]{Proposition}
\newtheorem{corollary}[theorem]{Corollary}

\theoremstyle{definition}
								
\newtheorem{definition}[theorem]{Definition}
\newtheorem{example}[theorem]{Example}

\newtheorem{remark}[theorem]{Remark}
\newtheorem*{remark*}{Remark}

\newcommand{\ZZ}{\mathbb{Z}}

\newcommand{\RR}{\mathbb{R}}

\renewcommand{\div}{\operatorname{div}}

\newcommand{\te}{\widetilde{e}}

\newcommand{\tg}{\widetilde{g}}
\renewcommand{\th}{\widetilde{h}}

\newcommand{\tv}{\widetilde{v}}
\newcommand{\tx}{\widetilde{x}}
\newcommand{\tC}{\widetilde{C}}

\newcommand{\tF}{\widetilde{F}}
\newcommand{\tG}{\widetilde{G}}

\newcommand{\tGa}{\widetilde{\Gamma}}

\newcommand{\ga}{\gamma}

\newcommand{\la}{\lambda}

\newcommand{\Ga}{\Gamma}
\newcommand{\La}{\Lambda}

\newcommand{\Si}{\Sigma}
\newcommand{\Om}{\Omega}

\newcommand{\fr}{\mathrm{fr}}
\newcommand{\free}{\mathrm{free}}

\DeclareMathOperator{\Ker}{Ker}
\DeclareMathOperator{\Coker}{Coker}
\DeclareMathOperator{\Hom}{Hom}

\DeclareMathOperator{\Vol}{Vol}

\let\Im\relax
\DeclareMathOperator{\Im}{Im}

\newcommand{\calB}{\mathcal{B}}
\newcommand{\calC}{\mathcal{C}}
\newcommand{\calD}{\mathcal{D}}

\newcommand{\calI}{\mathcal{I}}

\DeclareMathOperator{\Pic}{Pic}

\DeclareMathOperator{\trop}{trop}

\DeclareMathOperator{\rk}{rk}
\DeclareMathOperator{\Id}{Id}
\DeclareMathOperator{\Div}{Div}

\DeclareMathOperator{\Nm}{Nm}           

\DeclareMathOperator{\Prym}{Prym}       

\DeclareMathOperator{\ind}{ind}
\DeclareMathOperator{\mult}{mult}

\newcommand{\dil}{\mathrm{dil}}		
\newcommand{\ef}{\mathrm{ef}}
\newcommand{\ud}{\mathrm{ud}}	
\newcommand{\torf}{\mathrm{tf}}
\newcommand{\simp}{\mathrm{sim}}

\let\Rat\relax
\DeclareMathOperator{\Rat}{Rat}

\DeclareMathOperator{\Jac}{Jac}

\let\tilde\relax
\newcommand{\tilde}[1]{\widetilde{#1}}

\newcommand{\Felix}[1]{{\color{magenta}{\texttt Felix: #1}}}

\def\vertexsize {1.2pt}
\newcommand{\vertex}[3][1]{\fill (#2, #3) circle [radius = #1 * \vertexsize];}

\newcommand{\drawselfloop}[2]{
    \coordinate (A) at (#1,#2);
    \coordinate (c1) at (#1 + 2, #2 + 1);
    \coordinate (c2) at (#1 + 2, #2 - 1);
    
    \begin{scope}
        \clip (#1, #2 - 0.4) rectangle (#1 + 1.6, #2 + 0.4);    	
    	\draw[postaction = decorate]  (A) .. controls (c2) and (c1) .. (A);
    \end{scope}
}
\newcommand{\drawselfloopleft}[2]{
	\coordinate (A) at (#1,#2);
	\coordinate (c1) at (#1 - 2, #2 + 1);
	\coordinate (c2) at (#1 - 2, #2 - 1);
	\begin{scope}
        \clip (#1 - 1.6, #2 - 0.4) rectangle (#1, #2 + 0.4);    	
    	\draw[postaction = decorate]  (A) .. controls (c1) and (c2) .. (A);
    \end{scope}
}

\newcommand{\spiral}[2]{
    \begin{scope}
        \clip (#1, #2 - 0.4) rectangle (#1 + 1.7, #2 + 1.4);
    	\path[draw, postaction = decorate] (#1, #2+1) .. controls (#1+2.2, #2+2.4) and (#1+2.2, #2-1.4) .. (#1, #2);
    	\path[draw, postaction = decorate] (#1, #2) .. controls (#1+2, #2+2.2) and (#1+2, #2-1.2) .. (#1, #2+1);
    \end{scope}
}
\newcommand{\spiralleft}[2]{
    \begin{scope}
        \clip (#1 - 1.7, #2 - 0.4) rectangle (#1, #2 + 1.4);
    	\path[draw, postaction = decorate] (#1, #2+1) .. controls (#1-2.2, #2+2.4) and (#1-2.2, #2-1.4) .. (#1, #2);
    	\path[draw, postaction = decorate] (#1, #2) .. controls (#1-2, #2+2.2) and (#1-2, #2-1.2) .. (#1, #2+1);
    \end{scope}
}

\newcommand{\doublecover}[2]{%
	\begin{tikzcd}
		\begin{tikzpicture}[scale = 1.5]
			#1
		\end{tikzpicture} \arrow[d] \\
		\begin{tikzpicture}[scale = 1.5]
			#2
		\end{tikzpicture} 
	\end{tikzcd}
}

\newcommand{\orient}[1]{
    \begin{scope}[decoration={
			markings,
			mark= at position 0.6 with {\arrow{>}}}
            ]
        #1
    \end{scope}
}

\title[A matroidal perspective on the tropical Prym variety]{A matroidal perspective on the tropical Prym variety}

\author{Felix R\"ohrle}
\address{Universit\"at T\"ubingen, Fachbereich Mathematik, Auf der Morgenstelle 10, 72076 T\"ubingen, Germany}
\email{\href{mailto:roehrle@math.uni-tuebingen.de}{roehrle@math.uni-tuebingen.de}}
 
\author{Dmitry Zakharov}
\address{Department of Mathematics, Central Michigan University, Mount Pleasant, MI 48859, USA}
\email{\href{mailto:dvzakharov@gmail.com}{dvzakharov@gmail.com}}

\begin{document}
	\maketitle

\begin{center}
    \emph{to the memory of Igor Krichever}
\end{center}

\begin{abstract}
    We associate a matroid $M(\widetilde{\Gamma}/\Gamma)$ to a harmonic double cover $\pi:\widetilde{\Gamma}\to \Gamma$ of metric graphs. The matroid $M(\widetilde{\Gamma}/\Gamma)$ is a geometric interpretation of Zaslavsky's signed graphic matroid. We show that the principalization $\Prym_p(\widetilde{\Gamma}/\Gamma)$ of the tropical Prym variety of the double cover can be reconstructed from $M(\widetilde{\Gamma}/\Gamma)$, equipped with certain additional decorations. We describe the simplification of the matroid $M(\widetilde{\Gamma}/\Gamma)$ and show that the Prym variety does not change under simplification.
    
\end{abstract}
	
	\setcounter{tocdepth}{1}
	\tableofcontents








\section{Introduction}

The two most-studied objects in algebraic geometry are curves and abelian varieties. The construction of the Jacobian variety $\Jac(C)$ of an algebraic curve $C$, in the guise of abelian integrals, goes back to the foundations of algebraic geometry in the 19th century. In moduli, the association $C\mapsto \Jac(C)$ is called the \emph{Torelli morphism} $\tau_J:\mathcal{M}_g\to \mathcal{A}_g$, where $\mathcal{M}_g$ and $\mathcal{A}_g$ are the moduli spaces of smooth genus $g$ curves and dimension $g$ principally polarized abelian varieties (ppavs), respectively. The Torelli morphism is injective, and an effective description of its image is the still-unsolved \emph{Schottky problem}. 

There is another natural way to associate a ppav to an algebraic curve. Given an \'etale double cover $\widetilde{C}\to C$ of smooth curves of genera $2g-1$ and $g$, respectively, the \emph{Prym variety} $\Prym(\widetilde C/C)$ is the connected component of the identity of the kernel of the norm map $\Jac(\tC)\to \Jac(C)$. Prym varieties have been extensively studied since their reintroduction by Mumford in~\cite{1974Mumford}. The Prym variety carries a principal polarization that is half the polarization induced from $\Jac(\tC)$, and the assignment $(\tC\to C)\mapsto \Prym(\widetilde C/C)$ defines the \emph{Prym--Torelli map} $\tau_P:\mathcal{R}_g\to\mathcal \mathcal{A}_{g-1}$, where $\mathcal{R}_g$ is the moduli space of connected \'etale double covers of smooth genus $g$ curves. Unlike $\tau_J$, the morphism $\tau_P$ is never injective, and it is likewise an open problem to describe its image.



Tropical geometry is a modern branch of mathematics that studies certain combinatorial, piecewise-linear objects, which mimic concepts from algebraic geometry. A \emph{tropical curve} of genus $g$ is a metric graph $\Ga$ with $h_1(\Ga)=g$, a \emph{harmonic morphism} of tropical curves is a piecewise-linear map satisfying a certain balancing condition, and a \emph{tropical abelian variety} is a real torus carrying an auxiliary integral structure. By varying lengths and combinatorial types, we obtain moduli spaces $\mathcal{M}_g^{\trop}$ and $\mathcal{A}_g^{\trop}$ of tropical curves and tropical ppavs, respectively. 

We can naturally associate a \emph{tropical Jacobian} $\Jac(\Ga)$ to a tropical curve $\Ga$ (see~\cite{2008MikhalkinZharkov}). The corresponding tropical Torelli map $\tau_J^{\trop}:\mathcal{M}_g^{\trop}\to \mathcal{A}_g^{\trop}$ (which is not injective, unlike $\tau_J$) was investigated in~\cite{2010CaporasoViviani} and~\cite{2011BrannettiMeloViviani}. The key insight in these works is that the tropical Jacobian $\Jac(\Gamma)$ is in fact entirely governed by the graphic matroid $M(\Gamma)$ of $\Gamma$. It is therefore possible to study $\tau_J$, and in particular describe its fibers, from a purely matroidal perspective, and similarly to give a matroidal interpretation of the image of $\tau_J^{\trop}$. Indeed, the relevance of the graphic matroid in this context was well-known before the advent of tropical geometry, see e.g.~\cite{Gerritzen}. 

The tropical analogue of an \'etale double cover is a harmonic morphism $\tGa\to \Ga$ of metric graphs of degree two. The \emph{tropical Prym variety} $\Prym(\tGa/\Ga)$ of $\tGa\to \Ga$ was defined in~\cite{2018JensenLen}, in analogy with the algebraic setting, and is a polarized tropical abelian variety of dimension $g(\tGa)-g(\Ga)$. The tropical Prym variety was further investigated in~\cite{2021LenUlirsch},~\cite{2022LenZakharov},~\cite{2022RoehrleZakharov}, and~\cite{2023GhoshZakharov}.

\medskip

The purpose of this paper is to develop a matroidal perspective on the tropical Prym variety, with the aim of studying the tropical Prym--Torelli map. We show how to associate to a harmonic double cover $\pi : \tilde \Gamma \to \Gamma$ a dual pair of matroids $M(\tilde\Gamma / \Gamma)$ and $M^*(\tilde\Gamma / \Gamma)$ (see Definition~\ref{def:Prymmatroid}). These matroids turn out to be \emph{signed (co-)graphic matroids} in the sense of~\cite{1982Zaslavsky}, and the precise translation of Zaslavsky's definition into the language of double covers is given in Proposition~\ref{prop:Prymmatroid}. In Proposition~\ref{prop:signedgraphic} we give interpretations of the cryptomorphic descriptions of $M(\tilde \Gamma / \Gamma)$ in terms of the double cover $\tilde\Gamma \to \Gamma$. In Section~\ref{sec:recovery} we prove the main result of our paper and show that the matroid $M(\tGa/\Ga)$, suitably decorated, recovers all of the information of the Prym variety $\Prym(\tGa/\Ga)$ except for its polarization. We canonically associate a tropical ppav to $\Prym(\tGa/\Ga)$, the \emph{principalization} $\Prym_p(\tGa/\Ga)$, and signed cographic matroid recovers this object:


\begin{maintheorem}[Theorem~\ref{thm:Prym_from_matroid}]
	Let $\pi : \tilde \Gamma \to \Gamma$ be a harmonic double cover of metric graphs, and let $M = M(\tilde\Gamma / \Gamma)$ be the signed graphic matroid equipped with index function $\ind:\mathcal{I}(M^*)\to \ZZ_{>0}$ (Definition~\ref{def:index}), orientation $\overrightarrow{M}$ on $M$ induced by an orientation of $\pi$ (Section~\ref{subsec:orientedmatroids}), and edge length function $\ell:E(\Ga)\to \RR$. Then we can associate a tropical ppav $\Prym(M)$ to the matroidal data such that
	\[ \Prym_p(\tilde \Gamma / \Gamma) \cong \Prym(M). \]
\end{maintheorem}

It follows that two double covers with isomorphic decorated matroidal data have isomorphic Prym varieties, see Corollary~\ref{cor:criterion}. Example~\ref{ex:trees} shows that distinct double covers may have identical matroidal data and hence isomorphic Prym varieties.

For any matroid $M$, we can construct its \emph{simplification} $M_{\simp}$, having no circuits of size 1 or 2. In~\cite{2010CaporasoViviani} and~\cite{2011BrannettiMeloViviani} it is shown that simplifying the cographic matroid of $\Ga$ (contracting bridges and parallel edges, with an appropriate redistribution of edge lengths) does not change the Jacobian $\Jac(\Ga)$. In Section~\ref{sec:fibers}, we describe an analogous simplification procedure for double covers, and show that it does not change the Prym variety.


\begin{maintheorem}[Theorem~\ref{thm:simplification}] Let $\pi:\tGa\to \Ga$ be a double cover of metric graphs and let $\pi_{\simp}:\tGa_{\simp}\to \Ga_{\simp}$ be a simplification of $\pi$ (see Definition~\ref{def:simplification}). Then $\Prym_p(\tilde\Gamma / \Gamma) \cong \Prym_p(\tilde\Gamma_\simp / \Gamma_\simp)$.
    
\end{maintheorem}

In contrast to the case of metric graphs, we show by example that double covers with non-isomorphic simple matroids may have isomorphic Prym varieties.

In hindsight, it is clear that the signed graphic matroid $M(\tGa/\Ga)$ and its dual ought to have a close relation to the Prym variety $\Prym(\tGa/\Ga)$. For example, the papers~\cite{2022LenZakharov} and~\cite{2023GhoshZakharov} proved a formula for the volume of $\Prym(\tGa/\Ga)$ as a sum over certain subsets of $E(\Ga)$ called \emph{odd genus one decompositions} or \emph{ogods}, and these turn out to be the bases of $M^*(\tilde\Gamma / \Gamma)$ (see Section~\ref{subsec:volume} for details).

A natural question is to investigate the tropical Prym--Torelli morphism $\tau^{\trop}_P$ that associates to a harmonic double cover $\tGa\to \Ga$ its Prym variety $\Prym(\tGa/\Ga)$. There are two related issues with defining $\tau^{\trop}_P$. First, the polarization type of $\Prym(\tGa/\Ga)$ may change when the double cover $\tGa\to \Ga$ is deformed, specifically if the number of connected components of the dilation subgraph changes. Second, the behavior of $\Prym(\tGa/\Ga)$ is not continuous under such deformations. In future work, we plan to give an alternative definition of the tropical Prym variety that solves these issues, and investigate the redefined Prym--Torelli map from a matroidal perspective. The above Theorem shows that some of the non-injectivity of $\tau^{\trop}_P$ is purely matroidal in nature, and describing the remaining non-injectivity is the subject of ongoing research.

\medskip


We end by outlining a series of tropical problems inspired by Krichever's solution to the algebraic Schottky problems, which we now recall. Let $X$ be a ppav of dimension $g$ with principal polarization $\Theta$, and let $K(X)$ be the \emph{Kummer variety} of $X$, the image of $X$ under the map $\phi_{2\Theta}:X\to \mathbb{P}^{2^g-1}$. Fay's trisecant identity~\cite{2006Fay} states that if $X=\Jac(C)$ is a Jacobian, then $K(X)$ admits a four-parameter family of trisecant lines, indexed by quadruples of points of $C\subset\Jac(C)$. Gunning~\cite{1982Gunning} showed that the existence of such a family characterizes Jacobians among all ppavs. Welters~\cite{1983Welters},~\cite{1984Welters} strengthened this result by showing that Jacobians are characterized by the existence of a formal one-parameter family of flexes of $K(X)$ (limit trisecants when the three points are equal), and conjectured that the existence of a single trisecant of the Kummer variety characterizes Jacobians among ppavs. After a number of progressively stronger results, Krichever proved Welters's conjecture in~\cite{2006KricheverA} and~\cite{2010Krichever}. An analogous solution to the Prym--Schottky problem was given by Grushevsky and Krichever in~\cite{2010GrushevskyKrichever}. Beauville and Debarre~\cite{1987BeauvilleDebarre} showed that if $X=\Prym(\widetilde{C}/C)$ is a Prym variety, then $K(X)$ admits a four-parameter family of quadrisecant planes, and it was shown in~\cite{2010GrushevskyKrichever} that the existence of a symmetric pair of quadrisecants (but not a single quadrisecant) characterizes Prym varieties among all ppavs.

It is natural to formulate an analogous series of questions in the tropical setting. A theory of tropical theta functions has already been developed in~\cite{2008MikhalkinZharkov} and~\cite{foster2018non}. Fay's trisecant identity is based on an elementary property of divisorial rank (see Proposition 11.9.1 in~\cite{BirkenhakeLange}). It turns out that the analogous statement for tropical Baker--Norine rank is false, so it is not immediately clear what happens to Fay's identity under tropicalization. Nevertheless it is natural to ask whether the theta functions of a tropical Jacobian satisfy any identities, whether these identities have a geometric manifestation, whether they characterize tropical Jacobians among all tropical ppavs, and whether this characterization is computationally effective. Finally, all these questions can be asked about the tropical Prym variety as well.




\medskip

\textbf{Acknowledgments.} The authors would like to thank Kevin K\"uhn, Yoav Len, Dhruv Ranganathan, Victoria Schleis, Benjamin Schr\"oter, and Martin Ulirsch for insightful discussions. The first author gratefully acknowledges support from the SFB-TRR 195 ``Symbolic Tools in Mathematics and their Application'' of the German Research Foundation (DFG).

	\section{Setup and definitions}

We begin by reviewing a number of basic definitions concerning graphs, metric graphs, harmonic morphisms, double covers, and tropical abelian varieties.

\subsection{Graphs, harmonic morphisms, and double covers} \label{subsec:doublecovers}

A \emph{graph} $G$ consists of a set of \emph{vertices} $V(G)$, a set of \emph{half-edges} $H(G)$, a \emph{root map} $r:H(G)\to V(G)$, and a fixed-point-free involution $\iota:H(G)\to H(G)$. The \emph{edges} of $G$ are the orbits of $\iota$, with each edge consisting of two half-edges, and the set of edges is denoted $E(G)$. We allow graphs with loops and multiedges. The \emph{tangent space} to a vertex $v\in V(G)$ is the set $T_vG=r^{-1}(v)$ of half-edges rooted at $v$. The \emph{genus} $g(G)$ of a graph $G$ is its first Betti number:
\[
g(G)=|E(G)|-|V(G)|+1.
\]
Unless otherwise specified, we consider only finite connected graphs. Given a set of edges $F\subset E(G)$, we denote by $G\backslash F$ the (possibly disconnected) graph obtained by removing the edges in $F$, but retaining all vertices, even isolated ones. We denote $G[F]$ the minimal subgraph of $G$ containing all edges in $F$, in other words $G[F]$ is obtained from $G\setminus \big(E(G)\backslash F \big)$ by removing all isolated vertices. 

A \emph{morphism} of graphs $f:\tG\to G$ is a pair of maps $f:V(\tG)\to V(G)$ and $f:H(\tG)\to H(G)$ commuting with the root and involution maps, and thus inducing a map $f:E(\tG)\to E(G)$ on the edges. A \emph{harmonic morphism} of graphs is a pair $(f,d_f)$ consisting of a morphism of graphs $f:\tG\to G$ and a \emph{local degree} function $d_f:V(\tG)\cup H(\tG)\to \ZZ_{>0}$, such that $d_f(\th)=d_f(\th')$ for any edge $\te=\{\th,\th'\}\in E(\tG)$ and such that the local balancing condition
\begin{equation}
d_f(\tv)=\sum_{\th\in T_{\tv}\tG\cap f^{-1}(h)}d_f(\th)
\label{eq:harmonic}
\end{equation}
holds for each vertex $\tv\in V(\tG)$ and each half-edge $h\in T_{f(\tv)}G$. A harmonic morphism $f:\tG\to G$ to a connected target graph $G$ is surjective and has a well-defined \emph{global degree} given by
\[
\deg f =\sum_{\tv\in f^{-1}(v)}d_f(\tv)=\sum_{\th\in f^{-1}(h)}d_f(\th)
\]
for any $v\in V(G)$ or any $h\in H(G)$.

A harmonic morphism $p:\tG\to G$ of global degree two is called a \emph{double cover}, and we describe its structure in detail. A vertex $v\in V(G)$ is called \emph{dilated} if it has a unique preimage that we label $p^{-1}(v)=\{\tv\}$, with local degree $d_p(\tv)=2$. Similarly, a half-edge $h\in H(G)$ is \emph{dilated} if $p^{-1}(h)=\{\th\}$ with $d_p(\th)=2$, and a dilated edge $e=\{h,h'\}$ (consisting of dilated half-edges) has a unique edge $\te=\{\th,\th'\}$ lying over it. The root vertex of a dilated half-edge is dilated, so the dilated set forms a subgraph $G_{\dil}\subset G$ called the \emph{dilation subgraph}, which is isomorphic to its preimage $p^{-1}(G_{\dil})$. We say that $p$ is a \emph{free} double cover if $G_{\dil}=\emptyset$ and \emph{dilated} otherwise. We say that $p$ is \emph{edge-free} if $G_{\dil}$ consists of only vertices. The \emph{dilation index} of a double cover is defined to be
\[
d(\tG/G)=\begin{cases}
\mbox{number of connected components of }G_{\dil}, & \text{if } p\mbox{ is dilated,}\\
1, & \text{if } p\mbox{ is free.}
\end{cases}
\]

A vertex $v\in V(G)$ having two preimages is called \emph{undilated}, and we label its preimages $p^{-1}(v)=\{\tv^+,\tv^-\}$, where the local degrees are $d_p(\tv^{\pm})=1$. Similarly, a half-edge $h\in H(G)$ is \emph{undilated} if $p^{-1}(h)=\{\th^+,\th^-\}$ with $d_p(\th^{\pm})=1$. If the root vertex $v=r(h)$ is undilated, we label the preimages of $h$ so that $r(\th^{\pm})=\tv^{\pm}$ (if $v$ is dilated then $r(\th^{\pm})=\tv$ regardless of labeling). An undilated edge $e\in E(G)$ is called \emph{free} if both root vertices of $e$ (which may be the same) are undilated. The \emph{free subgraph} $G_{\fr}\subset G$ consists of the undilated vertices and the free edges, and the restriction $p|_{p^{-1}(G_{\fr})}:p^{-1}(G_{\fr})\to G_{\fr}$ of $p$ to the free subgraph is a degree two covering space in the topological sense.

The structure of the graph $\tG$ is determined by the involution map on the preimages $\th^{\pm}$ of the undilated half-edges, in other words by how the $\th^{\pm}$ pair up into edges. If the edge $e=\big\{h,\iota(h) \big\}\in E(G)$ is undilated but not free, then we can label the preimages of the half-edges in such a way that the edges of $\tG$ lying above $e$ are $\te^+=\big\{\th^+,\widetilde{\iota(h)}^+\big\}$ and $\te^-=\big\{\th^-,\widetilde{\iota(h)}^-\big\}$, in other words the involution is $\iota(\th^{\pm})=\widetilde{\iota(h)}^{\pm}$. If $e$ is free, however, then we record the involution by
\[
\sigma(e)=\begin{cases}
    +1, & \text{if } \iota(\th^{\pm})=\widetilde{\iota(h)}^{\pm},\\
    -1, & \text{if } \iota(\th^{\pm})=\widetilde{\iota(h)}^{\mp}.    
\end{cases}
\]
Given an orientation on $G$, we always give $\tG$ the induced orientation. We label the oriented preimages of an oriented edge $e=(h,\iota(h))$ as $\te^+=(\th^+,\iota(\th^+))$ and $\te^-=(\th^-,\iota(\th^-))$. 

The double cover $p:\tG\to G$ is thus determined by $G$, the choice of a dilation subgraph $G_{\dil}\subset G$ that also determines $G_{\fr}$, and a parity assignment $\sigma:E(G_{\fr})\to \{\pm 1\}$, in other words the structure of a \emph{signed graph} on $G_{\fr}$ (in fact, as we shall see in Proposition~\ref{prop:Prymmatroid}, the set of undilated edges $E(G)\backslash E(G_{\dil})$ carries the structure of a signed graph, in the more general sense of Zaslavsky~\cite{1982Zaslavsky}). Exchanging the preimages of an undilated vertex $v\in V(G_{\fr})$ swaps the signs of all free edges rooted at it, so the parity assignment $\sigma$ is determined up to \emph{vertex switching equivalence}. In modern terminology, the switching equivalence class of a parity assignment $\sigma$ is a well-defined element of the simplicial cohomology group $H^1(G_{\fr},\ZZ/2\ZZ)$, which is identified with the set of $\Hom \big( \pi_1(G_{\fr}),\,\ZZ/2\ZZ \big)$ of free double covers by covering space theory and the universal coefficient theorem.

For another perspective on double covers note that $p:\tG\to G$ determines an involution $i:\tG\to \tG$ that changes signs on undilated elements and fixes dilated elements. Conversely, given an involution $i:\tG\to \tG$, the quotient map $p:\tG\to \tG/i$ naturally has the structure of a double cover, with the dilation subgraph of $G$ corresponding to the fixed subgraph of $\tG$.


\subsection{Edge contractions} \label{subsec:edgecontraction} We extensively use the \emph{edge contraction} operation, which is a graph-theoretic deformation that allows us to move through the moduli spaces of graphs and double covers. Let $G$ be a graph, let $F\subset E(G)$ be a set of edges, and let $H = G[F]$ be the subgraph generated by these edges.  We denote by $G_F$ the graph obtained by contracting each connected component of $H$ to a separate vertex. An edge contraction induces a natural surjective contraction homomorphism $H_1(G,\ZZ)\to H_1(G_F,\ZZ)$.

An important property of contraction is that it can be performed for a harmonic morphism. Let $f:\tG\to G$ be a harmonic morphism of graphs, let $F\subset E(G)$ be a set of edges, and let $\widetilde{H} = \tilde G[\tilde F]$ be the subgraph of $\tG$ generated by $\widetilde{F}=f^{-1}(F)$. Consider the map $f_F:\tG_{\widetilde{F}}\to G_F$, where the vertex $\tv_i\in V(\tG_{\widetilde F})$ corresponding to a connected component $\widetilde{H}_i$ of $\widetilde{H}$ is mapped to the vertex corresponding to the connected component $f(\widetilde{H}_i)$ of $H$. It is elementary to check that $f_F$ is harmonic, provided that we set the local degree of $f_F$ at the vertex $\tv_i$ to be the global degree of $f$ on the subgraph $\widetilde{H}_i$. 

\begin{example} Let $p:\tG\to G$ be a double cover, let $F\subset E(G)$ be a set of edges, and assume for simplicity that the subgraph $H = G[F]$ generated by $F$ is connected. If $H$ has no dilated vertices, then the restricted cover $p|_{p^{-1}(H)}:p^{-1}(H)\to H$ is a free double cover, and $p^{-1}(H)$ has two connected components if and only if this double cover is trivial. In this case, the vertex $v\in V(G_F)$ corresponding to $H$ is undilated. However, if $p|_{p^{-1}(H)}$ is free but nontrivial, or if $H$ has at least one dilated vertex, then the vertex $v$ is dilated. In particular, the contraction of a free double cover need not be free.
\end{example}

\begin{example} \label{ex:resolution} 
We will frequently make use of the following construction, visualized in Figure~\ref{fig:free_resolution}. Let $p:\tG\to G$ be a dilated double cover. We construct the edge-free double cover $p_{\ef}:\tG_{\ef}\to G_{\ef}$ by contracting all dilated edges, in other words by contracting each connected component $G_i$ of $G_{\dil}$ to a dilated vertex $v_i$. We then resolve $p_{\ef}$ into a free double cover $p_{\free}:\tG_{\free}\to G_{\free}$ by adding a loop $e_i$ at each dilated vertex $v_i$ of $G_{\ef}$, and replacing the preimage vertex $\tv_i$ with two vertices $\tv^{\pm}_i$ joined by two parallel edges ${\te}^{\pm}_i$ mapping to the loop $e_i$. The cover $p_{\free}$ is not defined canonically (unlike $p_{\ef}$), since we must choose how to reattach the edges of $\tG_{\ef}$ to the new vertices $\tv^{\pm}_i$. The cover $p_{\ef}$ is obtained from $p_{\free}$ by contracting the new loops $e_i$. 
\end{example}

\begin{figure}
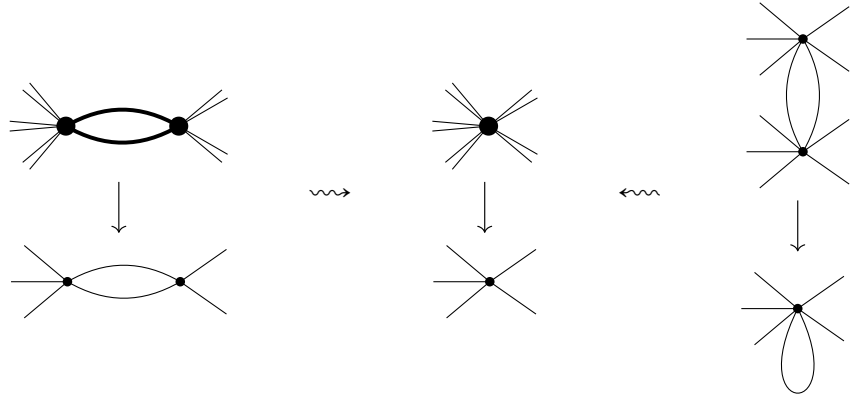

	\centering
	\include{free_resolution}
	\caption{Local picture of a dilated double cover $p$ (left) with the edge contractions producing $p_{\ef}$ (middle). On the right: $p_{\free}$ with its contraction back to $p_{\ef}$.}
	\label{fig:free_resolution}
\end{figure}

\subsection{Metric graphs}

Let $G$ be a graph and let $\ell:E(G)\to \RR_{>0}$ be an \emph{edge length} assignment. The pair $(G,\ell)$ defines a \emph{metric graph} $\Ga$, which is the topological space obtained by viewing each edge $e\in E(G)$ as a closed interval of length $\ell(e)$ and attaching the intervals according to the structure of $G$. We endow $\Ga$ with the shortest-path metric. The pair $(G,\ell)$ is called a \emph{model} for $\Ga$, and the genus of $\Ga$ is the genus of any underlying model. We will often abuse notation and refer to the edges and vertices of a metric graph $\Ga$, with respect to an implied choice of model. 

A \emph{harmonic morphism of metric graphs} $\phi:\tGa\to \Ga$ is a continuous, piecewise-linear map whose slopes $d_f(\te)$ along the edges $\te\in E(\tG)$ are positive integers satisfying the slope-balancing condition~\eqref{eq:harmonic}. In other words, there exist models $(\tG,\widetilde{\ell})$ and $(G,\ell)$ of $\tGa$ and $\Ga$, respectively, and a harmonic morphism of graphs $f:\tG\to G$ modeling $\phi$, such that $d_{\phi}(\te)=d_f(\te)$ for all $\te\in E(\tGa)$. The continuity of $\phi$ imposes the restriction
\[
\ell \big(\phi(\te) \big)=d_{\phi}(\te)\widetilde{\ell}(\te)
\]
for all edges $\te\in E(\tG)$. If $\phi:\tGa\to \Ga$ is a double cover of metric graphs, then $\phi$ is a factor two dilation along a dilated edge and an isometry along an undilated edge, explaining the terminology.

\subsection{Divisor theory on metric graphs} For the reader's convenience, we briefly recall the divisor theory on a metric graph $\Ga$. The divisor group $\Div(\Ga)$ is the free abelian group on the points of $\Ga$. A rational function on $\Ga$ is a real-valued piecewise-linear function with integer slopes, and the group of rational functions is denoted $\Rat(\Ga)$. The divisor of a rational function is
\[
\div:\Rat(\Ga)\longrightarrow \Div(\Ga),\qquad f\longmapsto\sum_{x\in \Ga}(\mbox{sum of incoming slopes at }x)\cdot x
\]
The Picard group $\Pic(\Ga)$ is the cokernel of $\div$. The degree map
\[
\deg:\Div(\Ga)\longrightarrow \ZZ,\qquad \sum_{x\in \Ga}a_x\cdot x\longmapsto \sum_{x\in \Ga}a_x
\]
descends to $\Pic(\Ga)$, and we denote $\Pic_d(\Ga)$ the set of equivalence classes of degree $d$ divisors. 


\subsection{Tropical abelian varieties} A \emph{real torus with integral structure} $\Sigma$ of dimension $g$, or an \emph{integral torus} for short, is defined by a triple $(\La,\La',[\cdot,\cdot])$, where $\La$ and $\La'$ are free abelian groups of rank $g$ and $[\cdot,\cdot]:\La\times \La'\to \RR$ is a nondegenerate pairing. The torus itself is $\Sigma=\Hom(\La,\RR)/\La'$, where $\La'$ is embedded in $\Hom(\La,\RR)$ via the assignment $\la'\mapsto [\cdot,\la']$. The integral lattice $\Hom(\La,\ZZ)$ in the universal cover $\Hom(\La,\RR)$ of $\Sigma$ may be viewed as a tropical analogue of a complex structure on a real torus.

Let $\Si_1=(\La_1,\La_1',[\cdot,\cdot]_1)$ and $\Si_2=(\La_2,\La_2',[\cdot,\cdot]_2)$ be integral tori, and let $f^\#:\La_2\to \La_1$ and $f_\#:\La_1'\to \La_2'$ be homomorphisms satisfying the relation
\[
\big[f^\#(\la_2),\la_1' \big]_1 = \big[\la_2,f_\#(\la_1')\big]_2
\]
for all $\la_2\in \La_2$ and $\la_1'\in \La_1'$. The Hom-dual $\Hom(\La_1,\RR)\to \Hom(\La_2,\RR)$ of $f^\#$ restricts to $f_\#$ on $\La_1'$, hence the pair $f=(f^\#,f_\#)$ defines a \emph{homomorphism of integral tori} $f:\Sigma_1\to \Sigma_2$. The maps $f^\#$ and $f_\#$ have the same rank, which is denoted $\rk f$ and which is equal to the dimension of $f(\Sigma_1)$. 

Let $f:\Sigma_1\to \Sigma_2$ be a homomorphism of integral tori. The group-theoretic kernel $\Ker f$ is the direct sum of a finite abelian group and an integral torus $(\Ker f)_0$, which we call the \emph{kernel torus} of $f$. It is defined by the triple
\[
    (\Ker f)_0  = \big( (\Coker f^\#)^\mathrm{tf}, \Ker f_\#, [\cdot, \cdot]_K \big),
\]
where for an abelian group $A$ we denote $A^\mathrm{tf}$ the quotient by its torsion subgroup,  and $[\cdot,\cdot]_K$ is the pairing $(\Coker f^\#)^\mathrm{tf}\times \Ker f_\#\to \RR$ induced by $[\cdot, \cdot]_1$. The pair $i=(i^\#,i_\#)$, where  $i_\# : \Ker f_\# \hookrightarrow \Lambda_1'$ is the natural inclusion and $i^\# : \Lambda_1 \twoheadrightarrow (\Coker f^\#)^\mathrm{tf}$ is the quotient map, defines the natural injective homomorphism of real tori $i : (\Ker f)_0 \to \Sigma_1$.

A \emph{polarization} on an integral torus $\Sigma=(\La,\La',[\cdot,\cdot])$ is a group homomorphism $\xi:\La'\to \La$ with the property that the induced bilinear form $(\cdot,\cdot)=[\xi(\cdot),\cdot]:\La'_\RR\times \La'_\RR\to \RR$ is symmetric and positive definite. A polarization is necessarily injective, and its \emph{type} is determined by the invariant factors of its Smith normal form. A polarization is \emph{principal} if it is bijective, in other words if its type is $(1,\ldots,1)$. An integral torus with a principal polarization is called a \emph{principally polarized tropical abelian variety}, or \emph{pptav} for short. Given an integral torus $\Sigma=(\Lambda,\Lambda',[\cdot,\cdot])$ with a polarization $\xi:\Lambda'\to \Lambda$, we define the \emph{principalization} to be the integral torus $\Sigma_p=(\Im \xi, \Lambda',[\cdot,\cdot])$ with the principal polarization $\xi:\Lambda'\to\Im \xi\cong \Lambda'$. The injection $j:\Im \xi\to \Lambda$ and the identity map $\Id:\Lambda'\to \Lambda'$ define a homomorphism $(j,\Id):\Sigma\to \Sigma_p$, which is a \emph{dilation} in the sense of Definition~4.6 of \cite{2022RoehrleZakharov}, so in particular it is a bijection on the underlying tori.\footnote{In Lemma~4.11 of \cite{2022RoehrleZakharov} we introduced a different notion of principalization $\Sigma^{\mathrm{pp}}$ which naturally comes with a dilation $\Sigma^{\mathrm{pp}} \to \Sigma$. The combined morphism $\Sigma^{\mathrm{pp}} \to \Sigma_p$ is of degree $2^g$ and is given by global multiplication by 2. In this article, we work with $\Sigma_p$ because it seems more natural from the matroidal perspective, but the reconstruction of the Prym variety in Section~\ref{sec:recovery} can be phrased for $\Prym^{\mathrm{pp}}$ as well, see also Remark~\ref{rem:reconstruction}.}

Let $f:\Sigma_1\to \Sigma_2$ be a homomorphism of integral tori, and let $\xi_2$ be a polarization on $\Sigma_2$. If $f$ is \emph{finite}, in other words if $\rk f=\dim \Sigma_1$, then the homomorphism $\xi_1=f^\#\circ \xi_2\circ f_\#$ determines an \emph{induced polarization} on $\Sigma_1$. As in the algebraic setting, a polarization induced from a principal polarization is not necessarily principal; we shall see that the Prym variety of a double cover is an  example of this. 

Let $\Sigma=(\Lambda,\Lambda',[\cdot,\cdot])$ be an integral torus of dimension $n$ with polarization $\xi$. The bilinear form $(\cdot,\cdot)=[\xi(\cdot),\cdot]$ on $\Lambda'$ extends to an inner product on the vector space $\Hom(\Lambda,\mathbb{R})\supset \La'$, which is the universal cover of $\Sigma=\Hom(\Lambda,\mathbb{R})/\La'$, and hence defines a translation-invariant Riemannian metric on $\Sigma$. The \emph{volume of $\Sigma$} is the volume of a fundamental parallelotope and is given by the Grammian determinant
\begin{equation}
\Vol^{2}(\Sigma)=\det \big( (\lambda_i', \lambda_j') \big)_{i,j},
\label{eq:Grammian}
\end{equation}
where $\lambda_1',\ldots,\lambda_n'$  is any basis of $\Lambda'$.

\subsection{The tropical Jacobian and the Abel--Jacobi map} Let $\Ga$ be a metric graph, and choose an oriented model $(G,\ell)$. The simplicial chain group $C_1(G,\ZZ)$ is the free abelian group on $E(G)$ and contains the simplicial homology group $H_1(G,\ZZ)$ which is of rank $g(\Ga)$. The chain groups naturally fit into a directed system with respect to refinements of models, and the images of the homology groups agree. We denote the direct limits $C_1(\Ga,\ZZ)$ and $H_1(\Ga,\ZZ)$, and refer the reader to~\cite{2011BakerFaber} for details. 

It is conceptually convenient to introduce the group $\Omega^1_{\Ga}$ of \emph{tropical 1-forms} on $\Ga$. This group is canonically isomorphic to $H_1(\Ga,\ZZ)$, and a cycle $\gamma=\sum \gamma(e)\cdot e\in H_1(\Ga,\ZZ)$ is written $\sum \gamma(e)de$ when it is viewed as a $1$-form. There is a natural \emph{integration pairing}
\begin{equation}
    [\cdot, \cdot] : \Omega^1_\Ga\times C_1(\Gamma, \ZZ)   \longrightarrow \RR, \qquad
    \left[\sum \gamma(e)de,\sum \delta(e) \cdot e\right] = \int_\gamma \omega=\sum \gamma(e)\delta(e)\ell(e),
    \label{eq:integrationpairing}
\end{equation}
which restricts to a perfect pairing $\Omega^1_\Ga\times H_1(\Gamma, \ZZ) \to \RR$. The \emph{Jacobian variety} of $\Ga$ is the $g(\Ga)$-dimensional pptav
\[
\Jac(\Gamma) = \big(\Omega^1_\Ga,\allowbreak \,H_1(\Gamma, \ZZ),\allowbreak \,[\cdot, \cdot]\big) =\Hom(\Om^1_{\Ga},\RR)\,/\,H_1(\Ga,\ZZ),
\]
where the principal polarization is the trivial isomorphism $H_1(\Gamma, \ZZ)=\Omega^1_\Ga$.

\begin{remark} We note that the integration pairing~\eqref{eq:integrationpairing}, and hence the induced inner product on the universal cover $\Hom(\Om^1_{\Ga},\RR)$ of $\Jac(\Ga)$, has a physical peculiarity: it is linear, rather than quadratic, in the edge lengths $\ell(e)$. Hence the units that are used to measure edge lengths on $\Ga$ become units of area from the point of view of the Riemannian geometry of $\Jac(\Ga)$.
    \label{rem:dimension}
\end{remark}

Let $q\in \Ga$ be a base point, and for $p\in \Ga$ let $\ga_p\in C_1(\Ga,\ZZ)$ denote any path from $q$ to $p$. The \emph{Abel--Jacobi map} $\phi_q:\Ga\to \Jac(\Ga)$ relative to $q$ is given by
\[
p\longmapsto \left(\omega\longmapsto \int_{\ga_p} \omega \right).
\]
The map $\phi_q$ extends to symmetric powers of $\Ga$ and thus to divisors. The map respects linear equivalence, and hence induces a map $\Pic_0(\Ga)\to \Jac(\Ga)$ that does not depend on the choice of $q$. The tropical Abel--Jacobi theorem (see Theorem 6.3 in~\cite{2008MikhalkinZharkov}) states that this map is an isomorphism. Under this identification, the Abel--Jacobi map sends $p$ to the divisor class $p-q$. 

\subsection{The tropical Prym and the Abel--Prym map} \label{subsec:Prym}  Let $\phi:\tGa\to \Ga$ be a harmonic morphism of metric graphs. We consider the natural pushforward and pullback homomorphisms 
\begin{align*}
    \phi^*=\Nm^\# :\Omega^1_\Ga &\longrightarrow \Omega^1_{\tGa}, & \phi_*=\Nm_\#:H_1(\tGa,\ZZ) &\longrightarrow H_1(\Ga,\ZZ), \\
    \sum_{e\in E(\Ga)} \gamma(e)de &\longmapsto \sum_{e\in E(\Ga)}\gamma(e)\sum_{\te\in \phi^{-1}(e)} d_{\phi}(\te) \, d\te, & \sum_{\te\in E(\tGa)} \gamma(\te)\cdot\te &\longmapsto \sum_{\te\in E(\tGa)} \gamma(\te)\cdot\phi(\te).
\end{align*}
The pair $\Nm=(\Nm^\#,\Nm_\#)=(\phi^*,\phi_*)$ defines the surjective \emph{norm homomorphism} $\Nm:\Jac(\tGa)\to \Jac(\Ga)$ of pptavs, which corresponds, under the identification of $\Jac$ with $\Pic_0$, to the direct image homomorphism
\[
\Nm:\Pic_0(\tGa)\longrightarrow \Pic_0(\Ga),\qquad \sum_{\tx\in \tGa}a_{\tx}\, \tx \longmapsto \sum_{\tx\in \tGa}a_{\tx}\, \phi(\tx).
\]
We now consider a double cover $\pi:\tGa\to \Ga$ of metric graphs of genera $\tg=g(\tGa)$ and $g=g(\Ga)$, respectively. Let
\[
\pi_*:H_1(\tGa,\ZZ)\longrightarrow H_1(\Ga,\ZZ),\qquad \pi^*:\Om^1_{\Ga}\longrightarrow \Om^1_{\tGa}
\]
denote the pushforward and pullback homomorphisms, denote
\[
K=(\Coker \pi^*)^{\torf},\qquad K'=\Ker \pi_*,
\]
and denote $[\cdot,\cdot]_K:K\times K'\to \RR$ the pairing induced by the integration pairing on $\Jac(\tGa)$. 
\begin{definition} The \emph{tropical Prym variety} $\Prym(\tGa/\Ga)$ of the double cover $\pi:\tGa\to \Ga$ is the kernel torus of the norm homomorphism:
\[
\Prym(\tGa/\Ga)=(\Ker \Nm)_0=\left(K,K',[\cdot,\cdot]_K\right)=\frac{\Hom(\Coker \pi^*,\RR)}{\Ker \pi_*}.
\]    
\end{definition}

The tropical Prym variety $\Prym(\tGa/\Ga)$ has dimension $h=\tg-g$ and is the connected component of the identity of the kernel $\Ker \Nm$ of the norm homomorphism $\Nm:\Jac(\tGa)\to \Jac(\Ga)$. The structure of the kernel is described by Proposition 6.1 in~\cite{2018JensenLen}. Let $i:\tGa\to \tGa$ denote the involution associated to $\pi$, then any element in $\Ker \Nm$ has a representative (viewed as an element of $\Pic_0(\tGa)$) of the form $E-i(E)$, where $E$ is an effective divisor $E$ on $\tGa$. If $\pi$ is a free double cover, then the parity of $\deg E$ does not depend on the choice of $E$, and $\Ker \Nm$ has two connected components distinguished by the parity. If $\pi$ is dilated, however, then $\Ker \Nm$ is connected.

The tropical Prym variety $\Prym(\tGa/\Ga)$ has a polarization induced from the principal polarization on $\Jac(\tGa)$. Its type was computed in Theorem 1.5.7 of~\cite{2021LenUlirsch} and Proposition 4.21 of~\cite{2022RoehrleZakharov} and is $(1,\ldots,1,2,\ldots,2)$ with $d-1$ many $1$'s, where $d$ is the dilation index of the double cover. For the most part of this article we will work with the principalization of the Prym variety, which we denote by $\Prym_p(\tGa / \Ga)$ and refer to as \emph{Prym variety} as well. For later reference, we spell out the defining data of $\Prym_p(\tGa / \Ga)$ and introduce the following notation. Any element of $\Ker \pi_*$ is contained in the $\ZZ$-span of elements $\tilde e^+ - \tilde e^-$ for $e$ ranging over the undilated edges of $\Gamma$. Hence we can view elements $\gamma \in \Ker \pi_*$ as \emph{chains} on $\Ga$ and write
\begin{equation}
\label{eq:chainsontarget}
\gamma=\sum_{e\in E(G)}\gamma(e)\cdot(\te^+-\te^-)= \sum_{e\in E(G)}\gamma(e)\cdot e\in C_1(G,\ZZ).
\end{equation}
Then $\Prym_p(\tGa/\Ga)$ is the pptav defined by the triple $(\Ker \pi_*,\Ker \pi_*,[\cdot,\cdot])$, where the pairing  is induced from $H_1(\tGa,\ZZ)$ and is given by
\begin{equation} \label{eq:pairingPrym}
    [\gamma, \gamma'] = \sum_{\tilde e \in E(\tGa)} \gamma(\tilde e) \gamma'(\tilde e) \ell(\tilde e) = \sum_{e \in E_{\ud}(\Ga)} 2\gamma(e) \gamma'(e) \ell(e) 
    \end{equation}
for all $\gamma, \gamma' \in \Ker\pi_*$, where the second equality uses $\ell(\tilde e^+) = \ell(\tilde e^-)=\ell(e)$.

We also define the tropical Abel--Prym map. Let $q\in \tGa$ be a base point, and for $p\in \tGa$ let $\ga_p\in C_1(\tGa,\ZZ)$ denote any path from $q$ to $p$. Let $i_*\ga_p$ denote the image of this path under the involution. The \emph{Abel--Prym map} $\psi_q:\tGa\to \Prym(\tGa / \Ga)$ is given by
\[
p\longmapsto \left(\omega\longmapsto \int_{\ga_p}\omega-\int_{i_*\ga_p}\omega\right).
\]
We note that $\psi_q=p-i(p)-(q-i(q))$ when viewed as an element of $\Pic_0(\tGa)$, hence lies in the connected component of the identity of $\Ker \Nm$ (which is not true of $p-i(p)$ if the double cover is free).

\subsection{Volumes of the Jacobian and the Prym} \label{subsec:volume}
One of the motivations for this paper are the formulas for the volume of the tropical Jacobian $\Jac(\Ga)$ found in~\cite{2014AnBakerKuperbergShokrieh} and the tropical Prym variety $\Prym(\tGa/\Ga)$ found in~\cite{2022LenZakharov} and~\cite{2023GhoshZakharov}. In both cases, the square of the volume is a polynomial in the lengths of the edges of $\Ga$. We shall see in the next section that the monomials are in fact indexed by the bases of an appropriate matroid on the set of edges of $\Ga$.

We recall the results of~\cite{2014AnBakerKuperbergShokrieh},~\cite{2022LenZakharov} and~\cite{2023GhoshZakharov}.

\begin{theorem}[Theorem 5.2 of~\cite{2014AnBakerKuperbergShokrieh}] Let $\Ga$ be a metric graph of genus $g$ with model $(G,\ell)$. The volume of the Jacobian of $\Ga$ is given by
\begin{equation}
\Vol^2 \big(\Jac(\Ga) \big) = \sum_{F\subset E(G)}\prod_{e\in F} \ell(e),
\label{eq:Jacvolume}
\end{equation}
where the sum is taken over all $g$-element subsets $F\subset E(G)$ such that $G\backslash F$ is a tree.    
\end{theorem}
The dimensional peculiarity noted in Remark~\ref{rem:dimension} makes its appearance: the square of the volume of $\Jac(\Ga)$ (which is a Riemannian manifold of dimension $g$) is a polynomial in the edge lengths of degree $g$, not the expected $2g$. The volume of the tropical Prym variety of a free and dilated double cover was computed in~\cite{2022LenZakharov} and~\cite{2023GhoshZakharov}, respectively.

\begin{theorem}[Theorem 3.4 of~\cite{2022LenZakharov} and Theorem 3.3 of~\cite{2023GhoshZakharov}] Let $\pi:\tGa\to \Ga$ be a double cover of metric graphs of genera $\tg$ and $g$, respectively, and let $h=\tg-g$. The volume of the Prym variety of $\pi$ is given by
\begin{equation}
\Vol^2 \big(\Prym(\tGa/\Ga) \big) = 2^{1-d(\tGa/\Ga)}\sum_{F\subset E_{\ud}(G)}4^{\ind(F)-1}\prod_{e\in F} \ell(e).
\label{eq:Prymvolume}
\end{equation}
Here the sum is taken over all $h$-element subsets of undilated edges $F\subset E_{\ud}(G)$ called ogods (see Definition~\ref{def:ogod}), the index $\ind(F)$ of an ogod $F$ is the number of connected components of $G\backslash F$, and $d(\tGa/\Ga)$ is the dilation index.
    
\end{theorem}

\section{The matroid of a double cover}
\label{sec:matroid}

The papers~\cite{2010CaporasoViviani} and~\cite{2011BrannettiMeloViviani} studied the tropical Torelli map that associates to a metric graph $\Ga$ its Jacobian $\Jac(\Ga)$, and showed that its fibers are most conveniently understood in terms of a combinatorial object, the \emph{graphic matroid} $M(\Ga)$ (we note that many of these results earlier appeared in German and without the use of the tropical language in~\cite{Gerritzen}). 
In this section, we describe an analogous object, the \emph{signed graphic matroid} $M(\tGa/\Ga)$ of a double cover $\pi:\tGa\to \Ga$, which plays a role in describing the fibers of the tropical Prym--Torelli map $(\tGa\to \Ga)\mapsto\Prym(\tGa/\Ga)$ in Section~\ref{sec:fibers}. This matroid was originally defined by Zaslavsky~\cite{1982Zaslavsky} in the framework of signed graphs, and we translate his definition into the language of double covers. We begin with the basic definitions.


\subsection{Matroids and graph theory}

A \emph{matroid} $M$ consists of a finite set $E(M)$, called the \emph{ground set} of $M$, and a family of subsets $\calI(M)$ of $E(M)$, called the \emph{independent sets}, satisfying the following axioms:

\begin{enumerate}
    \item $\emptyset \in \calI(M)$.

    \item If $A\subset B$ and $B\in \calI(M)$, then $A\in \calI(M)$.

    \item If $A,B\in \calI(M)$ and $A$ has more elements than $B$, then there exists $x\in A\backslash B$ such that $B\cup \{x\}\in \calI(M)$.
\end{enumerate}
The maximal elements of $\calI(M)$ are the \emph{bases} of the matroid, and the set of bases is denoted $\calB(M)$. All bases have the same cardinality, called the \emph{rank} of $M$. A subset of $E(M)$ that is not independent is called \emph{dependent}, and the set of dependent sets is denoted $\calD(M)$. A minimal dependent set is a \emph{circuit}, and the set of circuits is denoted $\calC(M)$. Each of the sets $\calI(M)$, $\calB(M)$, $\calD(M)$, and $\calC(M)$ determines the others and satisfies a set of properties that can be used to give an equivalent definition of a matroid (see Chapter~1 in~\cite{2011Oxley}). 

Starting with a matroid $M$, we construct several related matroids. The \emph{dual matroid} $M^*$ has the same ground set $E(M)$, and $B^*\subset E(M)=E(M^*)$ is a basis of $M^*$ if and only if $E(M)\backslash B^*$ is a basis of $M$. Given a subset $X\subset E(M)$, the \emph{restriction} of $M$ to $X$ is the matroid $M|_X$ with ground set $X$ and independent set $\calI(M|_X)=\{I\in \calI(M):I\subset X\}$. The \emph{deletion} of $X$ from $M$ is the restriction $M\backslash X=M|_{E(M)-X}$ to the complement. Finally, the \emph{contraction} of $M$ along $X$ is the matroid $M/X=(M^*\backslash X)^*$.

An \emph{$n$-circuit} of a matroid $M$ is a circuit with $n$ elements (a $1$-circuit is sometimes called a \emph{loop}, but we avoid this terminology). Two elements $e,f\in E(M)$ are \emph{parallel} if $\{e,f\}$ is a $2$-circuit, and a \emph{parallel class} of $M$ is a maximal subset of $F\subset E(M)$ containing no $1$-circuits and such that any $2$-element subset is a $2$-circuit. A parallel class is \emph{trivial} if it consists of a single element, and a matroid is called \emph{simple} if it has no non-trivial parallel classes. Given a matroid $M$, the \emph{simplification} of $M$ is the simple matroid $M_{\simp}$ constructed by deleting all $1$-circuits, and deleting all but one element from each parallel class. The simplification procedure is non-canonical, since it involves choosing a representative from each parallel class, but any two simplifications are isomorphic as matroids.

\begin{example} \label{ex:graphicmatroid} Let $G$ be a connected graph of genus $g$. The \emph{graphic matroid} $M(G)$ has ground set $E(G)$, and a subset $F\subset E(G)$ is independent if the subgraph $G[F]$ generated by $F$ is a forest (equivalently, if the homology group $H_1\big(G[F],\ZZ \big)$ is trivial). The bases of $M(G)$ are the spanning trees of $G$, the dependent sets are subgraphs having a nontrivial cycle, and the circuits are subgraphs consisting of a simple cycle. Given a set of edges $F\subset E(G)$, the deletion $M(G)\backslash F$ and contraction $M(G)/F$ are the graphic matroids of the graphs $G\backslash F$ and $G_F$, respectively. A $1$-circuit of $M(G)$ is a loop of $G$, and a parallel class is the set of edges between two vertices. The matroid $M(G)$ is simple if and only if the graph $G$ is simple, and $M(G)_{\simp}$ is the graphic matroid of the simplification of the graph $G$, obtained by deleting all loops, and deleting all but one edge in each set of multiedges between any pair of vertices. 


The dual matroid $M^*(G)$ is called the \emph{cographic} or \emph{bond matroid}. The independent sets of $M^*(G)$ are the sets of edges $F\subset E(G)$ whose removal does not disconnect $G$ (equivalently, such that the reduced homology group $\widetilde{H}_0(G\backslash F, \ZZ)$ is trivial). The bases of $M^*(G)$ are the complements of the spanning trees of $G$, and the rank of $M^*(G)$ is equal to $g$. A $1$-circuit of $M^*(G)$ is a bridge edge. A parallel class consists of a set of $n$ edges whose removal disconnects $G$ into $n$ connected components that are cyclically linked in $G$ by the removed edges. The simplification of $M^*(G)$ is obtained by contracting the bridge edges and all but one edge from each parallel class (since $M^*(G)$ is the dual, matroid-theoretic deletion is graph-theoretic contraction).

\end{example}

\subsection{The matroid of a double cover} In the seminal paper~\cite{1982Zaslavsky}, Zaslavsky introduces the \emph{signed graphic matroid} $M(G,\sigma)$ on the edge set $E(G)$ of a graph $G$ equipped with an edge parity assignment $\sigma:E(G)\to \{\pm 1\}$. As described in Section~\ref{subsec:doublecovers}, the pair $(G,\sigma)$ determines a free double cover $p:\tG\to G$, and Zaslavsky shows that the matroid $M(G,\sigma)$ is switching-equivalent, in other words depends only on the cover $p$ (Corollary 5.4 in~\cite{1982Zaslavsky}). He then considers the contraction of the matroid $M(G,\sigma)$ along a subset $F\subset E(G)$. Since the contracted double cover $p_F:\tG_{\tF}\to G_F$ may no longer be free, in order to stay within the framework of signed graphs, Zaslavsky is forced to introduce a more complicated category of graphs, having \emph{half-arcs} and \emph{free loops} in addition to ordinary edges (see Proposition~\ref{prop:Prymmatroid} for details).

In this section, we give an elementary geometric description of the signed graphic matroid and its dual in terms of double covers, which works equally well in the free and dilated cases. In addition, the description via double covers allows us to define an auxiliary function on the independent sets called the \emph{index} (see Definition~\ref{def:index}), which is crucial for reconstructing the Prym variety and which cannot be seen from Zaslavsky's approach. To motivate our definition, we recall that a set of edges of a graph $G$ is independent for the cographic matroid $M^*(G)$ if removing it does not disconnect $G$. We give an analogous definition for double covers.

\begin{definition} Let $p:\tG\to G$ be a double cover of graphs, with $G$ possibly disconnected. We say that $p$ is \emph{relatively connected} if the preimage of each connected component of $G$ is connected in $\tG$.  
\end{definition}

A double cover $p:\tG\to G$ of a connected graph $G$ is relatively connected if and only if it is not the trivial free double cover. If $e\in E(G)$ is a dilated edge, then $G\backslash \{e\}$ may or may not be connected, but each connected component has a dilated vertex, and hence the double cover obtained from $p$ by deleting $e$ from $G$ and its preimage from $\tG$ is still relatively connected. Removing an undilated edge, however, may relatively disconnect $p$. Hence we make the following definition.

\begin{definition} \label{def:Prymmatroid} Let $p:\tG\to G$ be a nontrivial double cover of a connected graph $G$. We define the matroid $M^*(\tG/G)$ of $p$ as follows: 
\begin{enumerate}
    \item The ground set of $M^*(\tG/G)$ is $E_{\ud}(G)=E(G)\backslash E(G_{\dil})$, the set of undilated edges of $G$.
    \item \label{item:indepdef} A set of undilated edges $F\subset E_{\ud}(G)$ is independent for the matroid $M^*(\tG/G)$ if the double cover $p|_{\tG\backslash p^{-1}(F)}:\tG\backslash p^{-1}(F)\to G\backslash F$ obtained by removing these edges from $G$ and their preimages from $\tG$ is relatively connected. Equivalently, $F$ is independent if the double cover $p|_{\tG\backslash p^{-1}(F)}$ induces an injective map on the reduced homology groups $\widetilde{H}_0 \big(\tG\backslash p^{-1}(F), \ZZ \big)\to \widetilde{H}_0(G\backslash F, \ZZ)$.
\end{enumerate}

\end{definition}

 We show in Proposition~\ref{prop:Prymmatroid} that $M^*(\tG/G)$ is a matroid, by identifying it with the signed graphic matroid of Zaslavsky. We first investigate when it is possible to remove a single undilated edge without relatively disconnecting the double cover. To simplify notation, given a double cover $p:\tG\to G$ and an undilated edge $e\in E_\ud(G)$, we denote $G^e=G\backslash\{e\}$, $\tG^e=\tG\backslash p^{-1}(e)$, and $p^e:\tG^e\to G^e$ the restriction of $p$. We first determine the relatively connected double covers that do not admit the removal of an edge.

\begin{definition} \label{def:elementarydouble}

An \emph{elementary double cover} $p:\tG\to G$ is one of the following two types (see Figure~\ref{fig:elementary_dc}):
\begin{enumerate}
    \item $p$ is dilated, $G$ and $G_{\dil}$ are connected, and $g(G_{\dil})=g(G)$.
    \item $p$ is free, $G$ is connected of genus one, and $p$ is not the trivial free double cover.
    
\end{enumerate}

\end{definition}

\begin{figure}
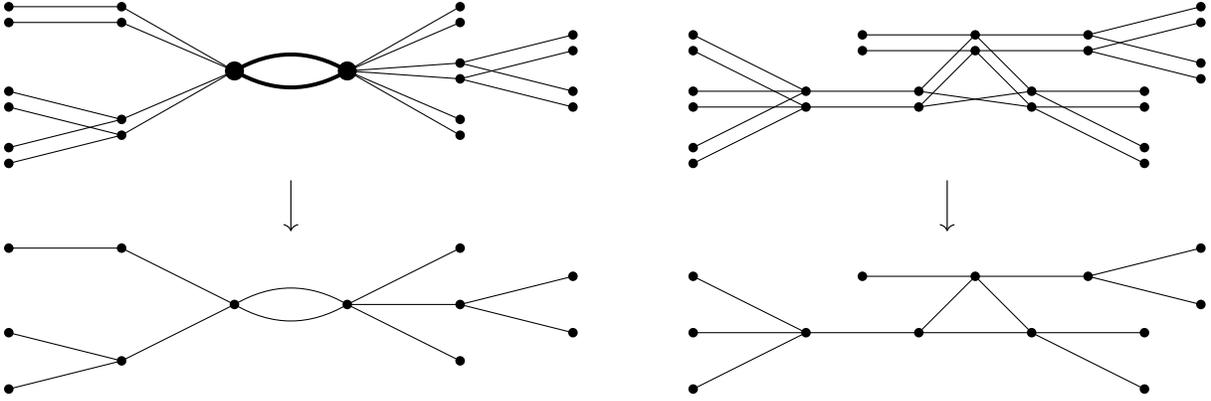

	\centering
	\include{elementary_double_covers}
	\caption{Elementary double covers of the first and second types.}
	\label{fig:elementary_dc}
\end{figure}

\begin{lemma} \label{lemma:elementarydoublecovers} Let $p:\tG\to G$ be an elementary double cover. Then $p$ is relatively connected, but for any undilated edge $e\in E(G)$ the double cover $p^e:\tG^e\to G^e$ is not relatively connected.
    
\end{lemma}

\begin{proof} If $p$ is of the first type, then $G$ consists of the dilation subgraph $G_{\dil}$ together with a number of undilated trees attached to it. Hence for any undilated edge $e$, the graph $G^e$ has two connected components, one of which is a tree with no dilation. Hence $p^e$ is not relatively connected since a tree has no nontrivial free double cover.

Similarly, if $p$ is free and $g(G)=1$, then $G^e$ is either a tree or a disjoint union of a genus one graph and a tree, and in both cases $p^e$ is not relatively connected.
\end{proof}

The next lemma shows that no other double covers have this property.

\begin{lemma} \label{lemma:nonelementary} Let $p:\tG\to G$ be a double cover of connected graphs of any of the following three types:

\begin{enumerate} \item $p$ is dilated and $G_{\dil}$ is not connected.

\item $p$ is dilated,  $G_{\dil}$ is connected, and $g(G)>g(G_{\dil})$.

\item $p$ is free and $g(G)\geq 2$.
    
\end{enumerate}
There exists an undilated edge $e\in E(G)$ such that $p^e:\tG^e\to G^e$ is relatively connected.
\label{lemma:nonelementarydoublecovers}
\end{lemma}

\begin{proof} If $G_{\dil}$ has two distinct connected components, choose a path between them consisting of undilated edges, and let $e$ be any edge on this path. The graph $G^e$ may be disconnected (if $e$ is a bridge), but each of its connected components has dilation, hence $p^e$ is relatively connected.

If $G_{\dil}$ is connected but $g(G)>g(G_{\dil})$, then the graph obtained by contracting $G_{\dil}$ to a single vertex has genus $g(G)-g(G_{\dil})>0$ and  hence a nontrivial cycle. Let $e$ be an (undilated) edge on this cycle. Removing $e$ from the original graph $G$ does not disconnect it, hence $p^e$ is relatively connected since $G^e$ has dilation. 

Finally, if $p$ is free, then there is a simple cycle $C$ on $G$ such that the restriction of $p$ to $C$ is not trivial. If $g(G)\geq 2$, we can also find a non-bridge edge $e\in G$ that does not lie on $C$. Removing $e$ does not disconnect $G$, and the cover $p^e$ remains relatively connected.
\end{proof}

We now prove that $M^*(\tG/G)$ is a matroid by showing that it is the dual of the signed graphic matroid of Zaslavsky. For the convenience of the reader, we give the bases of the matroid $M^*(\tG/G)$ as a separate definition.

\begin{definition} \label{def:ogod} Let $p:\tG\to G$ be a nontrivial double cover of a connected graph $G$, let $F\subset E_{\ud}(G)$ be a set of undilated edges, let $G\backslash F=G_1\cup\cdots\cup G_k$ be the decomposition into connected components, and let $p_i:p^{-1}(G_i)\to G_i$ be the restrictions of $p$. The set $F$ is called an \emph{odd genus one decomposition}, or \emph{ogod} for short, if each $p_i$ is an elementary double cover. 
    
\end{definition}

\begin{remark}\label{rem:ogod}
This terminology was introduced in~\cite{2022LenZakharov} in the context of free double covers. If $p:\tG\to G$ is free, then each $p_i:p^{-1}(G_i)\to G_i$ is a non-trivial double cover of a genus one graph by a genus one graph, corresponding to the nontrivial (i.e. \emph{odd}) element of $H^1(G_i,\ZZ/2\ZZ)=\ZZ/2\ZZ$. If $p$ is edge-free, then each $G_i$ may also be viewed as a genus one graph by assigning each dilated vertex an intrinsic genus. Hence the acronym \enquote{ogod}. The terminology breaks down for double covers with edge dilation, however, we generally do not consider such covers (see Lemma~\ref{lem:edgefree}).

\end{remark}

\begin{proposition} \label{prop:Prymmatroid} Let $p:\tG\to G$ be a nontrivial double cover of a connected graph $G$. Then $M^*(\tG/G)$ is a matroid of rank $g(\tG)-g(G)$ on the set $E_{\ud}(G)$ of undilated edges of $G$, whose bases are the ogods of $p$.

\end{proposition}


    
    
    
    

\begin{proof} It follows immediately from Lemmas~\ref{lemma:elementarydoublecovers} and~\ref{lemma:nonelementarydoublecovers} that a subset $F\subset E_{\ud}(G)$ is a basis, in other words a maximal independent set of $M^*(\tG/G)$, if and only if each $p_i$ is an elementary double cover. To show that $M^*(\tG/G)$ is a matroid, we show that its bases are the complements of the bases of the matroid $M(\Sigma)$ of a signed graph $\Sigma$ defined by Zaslavsky in~\cite{1982Zaslavsky}. We recall the definitions. Consider a tuple $\Sigma=(V,E,A,F,\sigma)$ of the following kind:

\begin{enumerate}
    \item $V$ and $E$ are respectively the vertex and edge sets of a graph $G'$, which is not necessarily connected. The set of connected components of $G'$ is denoted $\pi(\Sigma)$.
    \item $A$ is a set of \emph{half-arcs}, equipped with a root map $r:A\to V(G')$. An element $e\in A$ is an edge having one end rooted on the graph $G'$ and one loose end. When speaking of connected components of $G'$, we include the half-arcs rooted on it. 
    \item $F$ is a set of \emph{free loops}, which are edges with no root vertices.
    \item $\sigma:E\cup F\to \{\pm 1\}$ is a partial parity assignment on the edges, taking values $\sigma(f)=+1$ on all $f\in F$.
\end{enumerate}
The set of \emph{arcs} of $\Sigma$ is $E\cup A\cup F$ and it is the ground set of the matroid $M(\Sigma)$. A closed edge path $e_1\cdots e_n$ in $G'$ is called \emph{balanced} if $\sigma(e_1)\cdots\sigma(e_n)=1$. A set of arcs is called \emph{balanced} if it contains no half-arcs and if every closed path in it is balanced. The set of connected components of $G'$ whose arc set is balanced is denoted $\pi_b(\Sigma)$.

The bases of the matroid $M(\Sigma)$ are defined by Theorem 5.1 in~\cite{1982Zaslavsky} (we note that part (g) of Theorem 5.1, which specifically defines bases, is inaccurately stated, and was subsequently corrected in \cite{1983ZaslavskyErratum}). A set $S\subset E\cup A\cup F$ is a basis if the following conditions hold:

\begin{enumerate}
    \item For each balanced component $B\in \pi_b(\Sigma)$, the edges of $S$ in $B$ form a spanning tree for $B$.

    \item Each unbalanced connected component of the subgraph induced by $S$ is either a tree together with a unique half-arc, or a genus one graph with unbalanced cycle and no half-arcs.

    \item $S$ contains no free loops. 
\end{enumerate}

Now let $p:\tG\to G$ be a double cover. We define the associated tuple $\Sigma=(V,E,A,F,\sigma)$ as follows:
\begin{enumerate}
    \item The graph $G'$ is the free subgraph $G_{\fr}\subset G$, and $\sigma:E\to \{\pm 1\}$ is the parity assignmend defining the free double cover $p|_{p^{-1}(G_{\fr})}:p^{-1}(G_{\fr})\to G_{\fr}$ with respect to some choice of vertex labeling (see Section~\ref{subsec:doublecovers} for details).
    \item The half-arcs $A$ are the undilated edges having one free root vertex on $G_{\fr}$ and one dilated root vertex (which we discard).
    \item The free loops $F$ are the undilated edges both of whose root vertices are dilated.
\end{enumerate}
In other words, we delete the dilated vertices and edges from $G$, and leave dangling any undilated edges with missing root vertices, so that $E_{\ud}(G)=E\cup A\cup F$ according to whether an undilated edge loses none, one, or both root vertices.

Suppose further that $G$ is connected. If $p$ is free and nontrivial, then $G'=G$ is the unique (unbalanced) connected component, and there are no half-arcs or free loops. If $p$ is dilated, then each connected component of $G'$ has at least one half-arc (an undilated edge with one dilated and one undilated root vertex) or is a free loop. We see that $\pi_b(\Sigma)$ is empty in both cases.

Now let $F\subset E_{\ud}(G)$ be a basis of $M^*(\tG/G)$, let $G\backslash F=G_1\cup \cdots G_k$ be the decomposition into connected components, and let $p_i:p^{-1}(G_i)\to G_i$ be the corresponding elementary double covers. If $p_i$ is of the first type, then $G_i$ consists of its dilation subgraph $(G_i)_{\dil}$ with a number of attached undilated trees. Removing $(G_i)_{\dil}$, each tree acquires a unique half-arc, and no free loops are formed. If $p_i$ is of the second type, then $G_i$ is a genus one graph and its unique cycle is unbalanced because the free double cover $p_i$ is nontrivial. Hence $E_{\ud}(G)\backslash F$ is a basis of the signed graphic matroid $M(\Sigma)$. Reversing this construction, we see that all bases are obtained in this way. Hence $M^*(\tG/G)$ is the dual matroid to $M(\Sigma)$.

To complete the proof, we show that the rank of $M^*(\tG/G)$ is equal to $g(\tG)-g(G)$. We use the construction of Example~\ref{ex:resolution}. Let $p_{\ef}:\tG_{\ef}\to G_{\ef}$ be the edge-free double cover obtained from $p:\tG\to G$ by contracting each of the $n$ connected components of $G_{\dil}$ to a separate dilated vertex. The matroids $M^*(\tG/G)$ and $M^*(\tG_{\ef}/G_{\ef})$ are canonically isomorphic. Indeed, $G$ and $G_{\ef}$ have the same set of undilated edges, and for any $F\subset E_{\ud}(G)=E_{\ud}(G_{\ef})$, the connected componets of $G_{\ef}\backslash F$ are obtained from the connected components of $G\backslash F$ by contracting the dilated edges. Hence $F$ is independent for $M^*(\tG/G)$ if and only if it is independent for $M^*(\tG_{\ef}/G_{\ef})$. Furthermore, $g(\tG_{\ef})-g(G_{\ef})=g(\tG)-g(G)$, since both $\tG$ and $G$ lose genus equal to the total genera of the components of $G_{\dil}$.

Now let $p_{\free}:\tG_{\free}\to G_{\free}$ be the free double cover obtained by resolving the $n$ dilated vertices of $G_{\ef}$. Then $g(G_{\free})=g(G_{\ef})+n$ and $g(\tG_{\free})=g(\tG_{\ef})+n$, so $g(\tG_{\free})-g(G_{\free})=g(\tG_{\ef})-g(G_{\ef})$. Furthermore, under the natural inclusion $E_{\ud}(G_{\ef})\subset E_{\ud}(G_{\free})$, any basis of $M^*(\tG_{\ef}/G_{\ef})$ is also a basis of $M^*(\tG_{\free}/G_{\free})$. Hence it is enough to prove the assertion for the free double cover $p_{\free}:\tG_{\free}\to G_{\free}$. Let $F\subset E(G_{\free}) = E_{\ud}(G_{\free})$ be a basis of $M^*(\tG_{\free}/G_{\free})$, then the connected components $G_{\free}\backslash F=G_1\cup\cdots\cup G_k$ all have genus one. Since the quantity $g(\cdot)-1=|E(\cdot)|-|V(\cdot)|$ is additive in connected components, it follows that 
\[
\rk M^*(\tG/G)=|F|= g(G_{\free})-1-\sum_{i=1}^k\big(g(G_i)-1\big)=g(G_{\free})-1=g(\tG_{\free})-g(G_{\free})=g(\tG)-g(G).
\]
%
%
%
%
%
\end{proof}

As an illustration, and for future reference, we describe the $1$-circuits and $2$-circuits of the signed cographic matroid $M^*(\tG/G)$. A $1$-circuit is an undilated edge $e\in E_{\ud}(G)$ whose removal relatively disconnects the double cover. Figure~\ref{fig:1circuits} shows the two topological types of $1$-circuits of $M^*(\tG/G)$, depending on whether or not removing $e$ disconnects the target graph $G$. In the second diagram, the edge $e$ may or may not be a loop.

\begin{figure}
	\centering
	\input{1circuits}
	\caption{1-circuits of $M^*(\tG / G)$. Boxes denote arbitrary connected graphs. The curved edge on the target may or may not be a loop. }
	\label{fig:1circuits}
\end{figure}
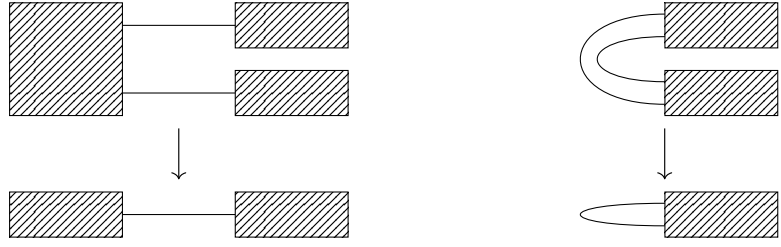

Similarly, a $2$-circuit is a subset $\{e,f\}\subset E_{\ud}(G)$ whose removal relatively disconnects $p$ but such that neither $\{e\}$ nor $\{f\}$ is a $1$-circuit. There are four topological types of $2$-circuits, illustrated on Figure~\ref{fig:2circuits}, where the curved edges may or may not represent loops. We label the edges of the $2$-circuit with multiplicities that are used to define the $2$-simplification of the matroid $M^*(\tG/G)$ (see Definition~\ref{def:simplification_mult} for details, or simply note that an edge is labeled by a 2 if it is a bridge and 1 otherwise). 

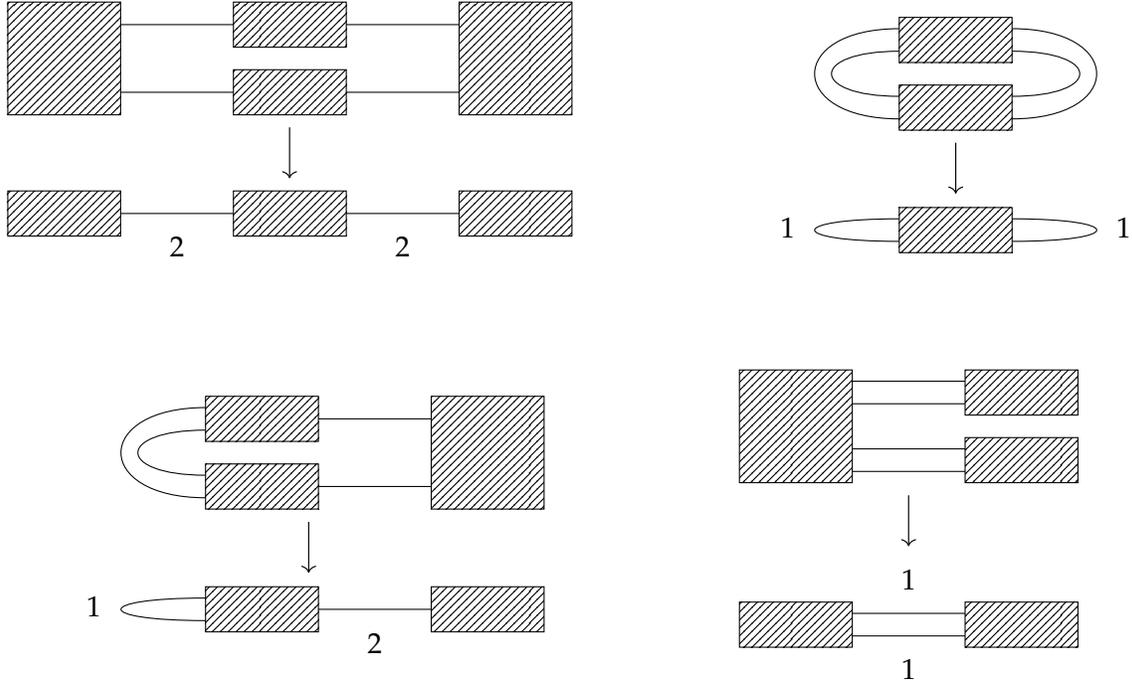
\begin{figure}
	\centering
	\input{2circuits}
	\caption{2-circuits of $M^*(\tilde G / G)$. Boxes denote arbitrary connected graphs. Numbers indicate multiplicities (see Definition~\ref{def:simplification_mult}).}
	\label{fig:2circuits}
\end{figure}


\subsection{The dual matroid $M(\tG/G)$} In this section, we describe the \emph{signed graphic matroid} $M(\tG/G)$, which is the dual of the signed cographic matroid $M^*(\tG/G)$, in terms of the double cover $\tilde G \to G$. We first give a description of its bases, independent sets, and circuits. As motivation, recall that a set of edges $F\subset E(G)$ of a graph $G$ is independent for the graphic matroid if the subgraph $G[F]$ induced by $F$ is a forest, or equivalently if the group $H_1 \big(G[F],\ZZ \big)$ is trivial. We now show that a similar statement holds for $M(\tG/G)$ as well.

\begin{proposition} \label{prop:signedgraphic} Let $p:\tG\to G$ be a double cover and let $F\subset E_{\ud}(G)$ be a set of undilated edges. Let $G[F]=G_1\cup\cdots\cup G_k$ be the connected component decomposition of the subgraph $G[F]$ generated by $F$, and let $p[F]:p^{-1}(G[F])\to G[F]$ and $p_i:p^{-1}(G_i)\to G_i$ denote the restrictions of $p$.

\begin{enumerate}
    \item $F$ is a basis of $M(\tG/G)$ if and only if $G[F]$ contains all undilated vertices of $G$ and each $p_i$ is an elementary double cover. \label{item:Mbasis} 

    \item $F$ is independent if and only if the map $p[F]_*:H_1 \big(p^{-1}(G[F]),\ZZ \big)\to H_1 \big(G[F],\ZZ \big)$ induced by $p[F]$ is injective. \label{item:Mindep}
       
    
    \item $F$ is a circuit if and only if $G[F]=G_1$ is connected, has no undilated vertices of valency one, and either $g(G[F])=1$ and $p[F]$ is the trivial free double cover (type I on Figure~\ref{fig:circuits}), or $p[F]$ is one of types II-VI shown on Figure~\ref{fig:circuits}. 
    
\end{enumerate} 
    
\end{proposition}

\begin{proof} The bases of $M(\tG/G)$ are the complements of the bases of $M^*(\tG/G)$ in the set of undilated edges $E_{\ud}(G)$, so Part~(\ref{item:Mbasis}) is simply a restatement of Proposition~\ref{prop:Prymmatroid}. To prove Part~(\ref{item:Mindep}), we note that $(p_i)_*$ is injective if $p_i$ is an elementary double cover, and that this property is preserved if additional edges are removed. Hence $p[F]_*$ is injective for any independent set $F$. Conversely, it is easy to see (for example, from the explicit description of the kernel given in Propositions~4.19 and~4.20 in~\cite{2022RoehrleZakharov}) that if $(p_i)_*$ is not injective, then either $p_i$ is a trivial double cover of a graph $G_i$ of genus $g(G_i)\geq 1$, or else it is one of the three types listed in Lemma~\ref{lemma:nonelementary}. In either case, it is clear that $F$ is not contained in a basis of $M(\tG/G)$, hence it is dependent. 

We now describe the circuits of $M(\tG/G)$. Looking at each of the six double covers shown in Figure~\ref{fig:circuits}, we see that the pushforward map on homology is not injective, but becomes so if any edge is removed. Hence each of the double covers represents a minimal dependent set of $M(\tG/G)$, in other words a circuit. 

Conversely, let $F$ be a dependent set of $M(\tG/G)$. We remove edges from $F$ to find a minimal dependent set. Since $p[F]_*$ is not injective, we can assume without loss of generality that $(p_1)_*$ is not injective and replace $F$ by $E(G_1)$. If $G[F]$ has an undilated cycle with disconnected preimage, then we can remove all other edges and obtain a circuit of type I. If it does not, then it is one of the three types listed in Lemma~\ref{lemma:nonelementary}, and it is then elementary to verify that the circuits of types II-VI are the minimal double covers of the types given in the lemma.
\end{proof}

Figure~\ref{fig:circuits} shows the six topological types of circuits of the signed graphic matroid $M(\tG/G)$. Undilated vertices of valency 2 are not shown, so a single edge in the picture may represent a chain of edges in $G$. Bold dots represent dilated vertices. In each type, the pushforward map $p[C]_*:H_1\big(p^{-1}(G[C]),\ZZ\big)\to H_1\big(G[C],\ZZ\big)$ has one-dimensional kernel, and the indicated cycle $\gamma_C \in H_1\big( p^{-1}(G[C]), \ZZ \big)$ is a generator.

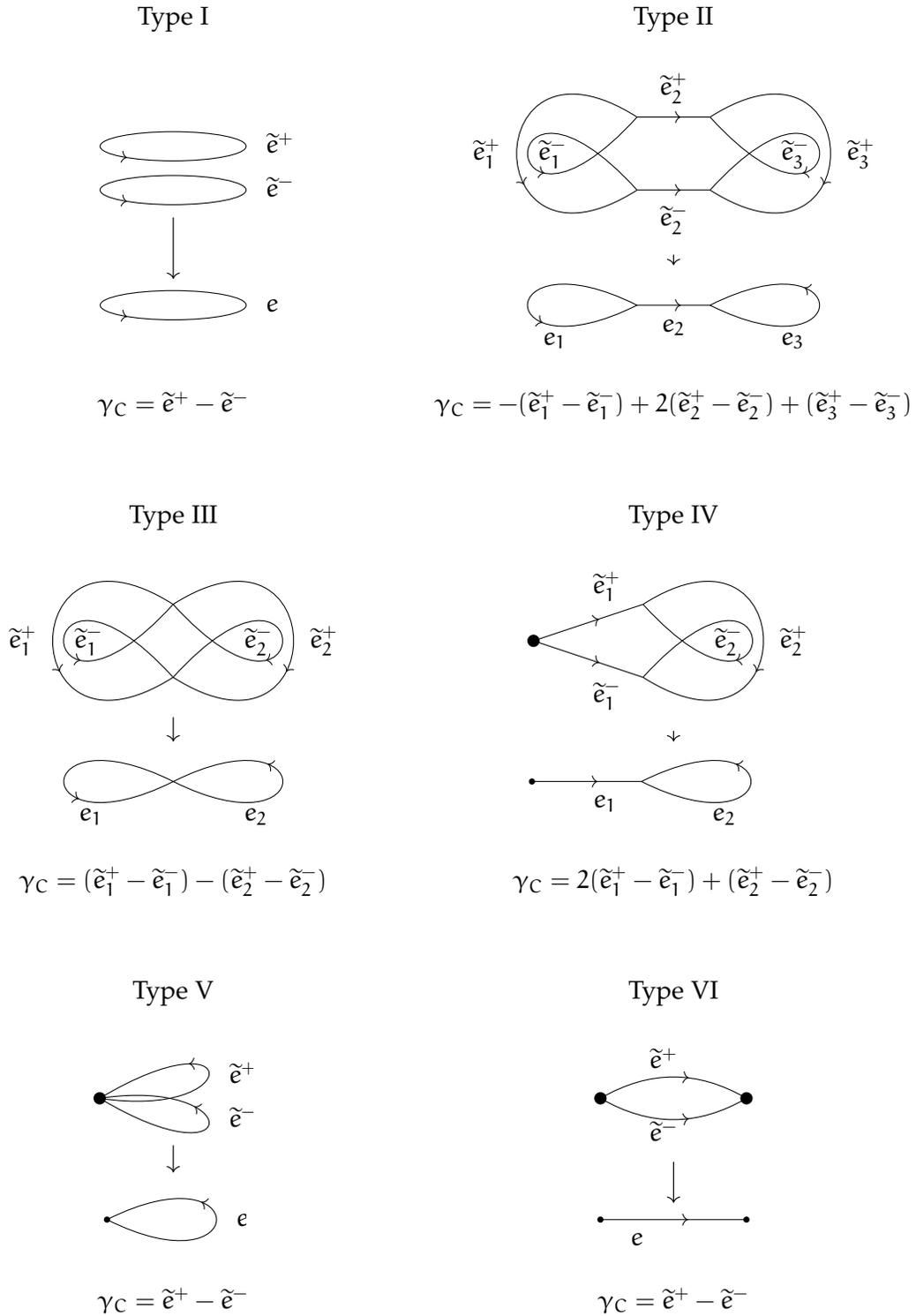
\begin{figure}[!p]
	\centering
	\input{circuits}
	\caption{Circuits of the signed graphic matroid. Undilated vertices of valency 2 are not shown. Bold dots represent dilated vertices. The cycle $\gamma_C$ is the fundamental cycle associated to $C$ (see Definition~\ref{def:fundcycle}).}
	\label{fig:circuits}
\end{figure}


\section{Recovering the Prym variety from the matroid}
\label{sec:recovery}

In this section, we explain how to reconstruct the principalized Prym variety $\Prym_p(\tGa / \Ga)$ of a double cover $\pi:\tGa\to \Ga$ of metric graphs from the signed graphic matroid equipped with certain additional decorations. We always implictly choose a model for the double cover, and by abuse of notation denote its signed graphic and cographic matroid by $M(\tGa/\Ga)$ and $M^*(\tGa/\Ga)$. To motivate the reconstruction procedure, we first recall how to reconstruct the Jacobian variety of a metric graph from the graphic matroid.

\subsection{Matroidal reconstruction of Jacobian and Prym: an overview} Let $\Ga$ be a metric graph. Choose an orientation for $\Ga$. Given a circuit $C$ of the graphic matroid $M(\Ga)$, there exists a simple cycle $\delta_{C}=\sum \delta_{C}(e)\cdot e$ in  $H_1(\Ga,\ZZ)$ supported on $C$ with $\delta_C(e) = \pm 1$ for all edges $e \in C$. The cycle $\delta_C$ is unique up to sign. While the support of $\delta_C$ can be determined from $M(\Ga)$, the signs of its coefficients depend on the chosen orientation and are determined by the condition that $\delta_{C}$ is a closed cycle. This information is not contained in $M(\Ga)$. 

To make the construction fully matroid-theoretic, we use the chosen orientation on $\Ga$ to define the structure of an \emph{oriented matroid} $\overrightarrow{M}(\Ga)$ on $M(\Ga)$ (see Example~\ref{ex:orientedgraphic} or~\cite[Section~1.1]{Oriented_matroids}). The coefficients of the cycles $\delta_{C}$ are then simply the signs with which the edges of $C$ occur in an oriented circuit $\overrightarrow{C}$ of $\overrightarrow{M}(\Ga)$ with underlying (unoriented) circuit $C$. 

It is an elementary fact that the cycles $\big\{\delta_{C} : C\in \calC \big(M(\Ga) \big) \big\}$ are a spanning set for $H_1(\Ga,\ZZ)$. Moreover, we can choose a basis as follows. Let $B=\{e_1,\ldots,e_g\}$ be a basis of the cographic matroid $M^*(\Ga)$, so that $T=\Ga\backslash B$ is a spanning tree. For each $i=1,\ldots,g$ let $C_i$ be the unique circuit of $M(\Ga)$ supported on $T\cup \{e_i\}$, then the cycles $\delta_{C_i}$ form a basis for $H_1(\Ga,\ZZ)$. 
Therefore, we can associate a pptav $\Jac \big(\overrightarrow{M}(\Ga)\big)$ to the oriented matroid $\overrightarrow{M}(\Ga)$ equipped with the edge length function $\ell:E(\Ga)\to \RR_{>0}$, and $\Jac\big(\overrightarrow{M}(\Ga)\big)$ is isomorphic to the Jacobian $\Jac(\Ga)$.

However, even more is true. The graph Jacobian $\Jac(\Ga)$ does not depend on the choice of orientation of $\Ga$. By~\cite[Corollary~6.2.8]{1978BlandLasVargnas}, graphic matroids have a unique reorientation class, hence any oriented matroid having underlying matroid $M(\Ga)$ is in fact obtained from an orientation of $\Ga$. Therefore, we can construct $\Jac \big(\overrightarrow{M}(\Ga)\big)$ directly from the matroid $M(\Ga)$ by choosing any structure of an oriented matroid, completely bypassing graphs. The resulting object, which we denote $\Jac(M(\Ga))$, is isomorphic to the Jacobian $\Jac(\Ga)$ of any graph whose matroid is isomorphic (in a way that preserves edge lengths) to $M(\Ga)$.
%
%
%
%
%
We summarize the reconstruction procedure in the following diagram: 
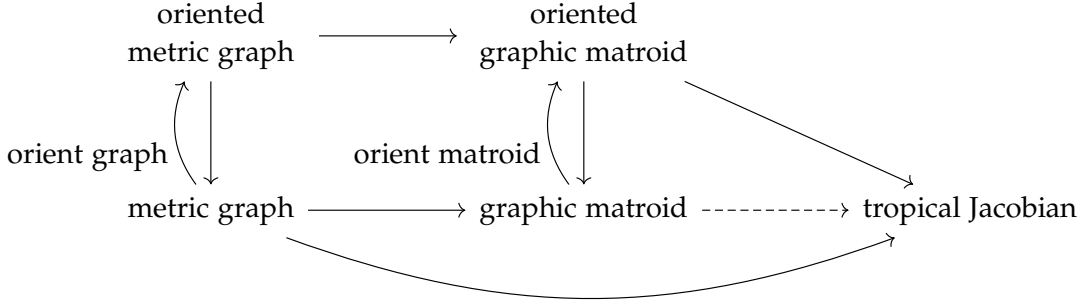
\begin{center}
    \begin{tikzcd}[row sep = huge, column sep = huge]
    \begin{minipage}{2.5cm}
        \centering
        oriented \\ metric graph
    \end{minipage} \arrow[r] \arrow[d] & \begin{minipage}{3cm}
        \centering
        oriented \\ graphic matroid
    \end{minipage} \arrow[d] \arrow[dr] & \\
    \mbox{metric graph} \arrow[u, bend left = 30, "\mbox{orient graph}"] \arrow[r] \arrow[rr, bend right = 20] & \mbox{graphic matroid} 
     \arrow[u, bend left = 30, "\mbox{orient matroid}"]\arrow[r, dashed] & \mbox{tropical Jacobian} 
    \end{tikzcd}
\end{center}

In this section, we describe an analogue of the fundamental cycle construction to determine the principalized Prym variety $\Prym_p(\tGa/\Ga)$ of a double cover $\pi:\tGa\to \Ga$ from the signed graphic matroid $M(\tGa/\Ga)$ and its dual $M^*(\tGa/\Ga)$:
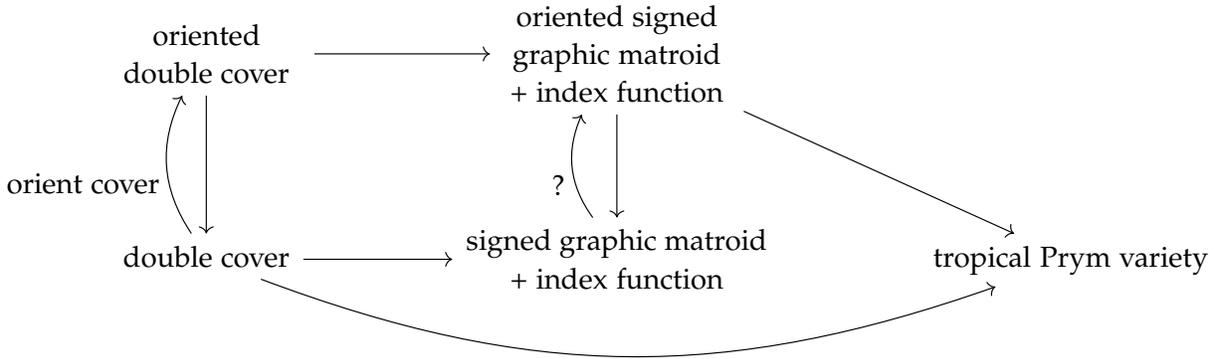
\begin{center}
    \begin{tikzcd}[row sep = huge, column sep = huge]
    \begin{minipage}{2.5cm}
        \centering
        oriented \\ double cover
    \end{minipage} \arrow[r] \arrow[d] & \begin{minipage}{3cm}
        \centering
        oriented signed \\ graphic matroid \\
        + index function
    \end{minipage} \arrow[d] \arrow[dr] & \\
    \mbox{double cover} \arrow[r] \arrow[u, bend left = 30, "\mbox{orient cover}"] \arrow[rr, bend right = 20] & \begin{minipage}{4cm}
        \centering
        signed graphic matroid \\ + index function
    \end{minipage}  \arrow[u, bend left = 30, "\mbox{?}"]
    & \mbox{tropical Prym variety} 
    \end{tikzcd}
\end{center}
First, it is necessary to equip $M^*(\tGa/\Ga)$ with an auxiliary index function, this is explained in Section~\ref{subsec:index}. We then explain in Section~\ref{subsec:orientedmatroids} how an orientation of the double cover determines an orientation on $M(\tGa/\Ga)$ and $M^*(\tGa/\Ga)$. We then construct the Prym variety $\Prym_p(\tGa/\Ga)$ from the oriented matroid $\overrightarrow{M}(\tGa/\Ga)$, the index function, and the edge length function. Unlike the case of graphs, it is not true that all orientations of $M(\tGa/\Ga)$ come from a choice of orientation on the double cover (for example, we show in Section~\ref{subsec:isomPryms} that $U_{2,4}$ is a signed graphic matroid, and \cite[Example~1]{1995GelfandRybnikovStone} shows that $U_{2,4}$ has three reorientation classes). Hence we are not presently able to fully reconstruct $\Prym_p(\tGa/\Ga)$ solely from the matroid $M(\tGa/\Ga)$, we still need to retain the reorientation class that is determined by the double cover. In addition, we are not currently able to recover the induced polarization and reconstruct the non-principally polarized $\Prym(\tGa/\Ga)$.

\subsection{The fundamental cycle construction}\label{subsec:fundcycle} In this section, we describe an analogue of the fundamental cycle construction for double covers. We first observe that the dilated edges of $G$ do not contribute to the matroid $M(\tG/G)$, and we show that they are not seen by the Prym variety either. 

\begin{lemma} \label{lem:edgefree} 
Let $\pi : \tilde\Gamma \to \Gamma$ be a double cover of metric graphs and let $\pi_{\ef} : \tilde\Gamma_{\ef} \to \Gamma_{\ef}$ be the edge-free double cover obtained by contracting all dilated edges of $\Ga$ (see Example~\ref{ex:resolution}). Then the signed graphic matroids and the Prym varieties of the two double covers are isomorphic:
\[
M(\tGa/\Ga)\cong M(\tGa_{\ef}/\Ga_{\ef}),\qquad \Prym(\tGa/\Ga)\cong \Prym(\tGa_{\ef}/\Ga_{\ef}),\qquad\Prym_p(\tGa/\Ga)\cong \Prym_p(\tGa_{\ef}/\Ga_{\ef}).
\] 
\end{lemma}

\begin{proof} We already established that $M(\tGa/\Ga)\cong M(\tGa_{\ef}/\Ga_{\ef})$ in the last part of the proof of Proposition~\ref{prop:Prymmatroid}. 
To prove that the Pryms are isomorphic, we consider what happens when a single dilated edge $e$ is contracted. Let $\pi_e:\tGa_e\to \Ga_e$ be the double cover obtained by contracting $e$.
There are natural surjective contraction maps $H_1(\tGa,\ZZ)\to H_1(\tGa_e,\ZZ)$ and $H_1(\Ga,\ZZ)\to H_1(\Ga_e,\ZZ)$, given by setting the corresponding edges to zero. If $\gamma$ is a cycle in $\Ker \big(H_1(\tGa,\ZZ)\to H_1(\Ga,\ZZ) \big)$, then $e$ does not occur in $\gamma$, hence $\gamma$ corresponds to an element of $\Ker \big(H_1(\tGa_e,\ZZ)\to H_1(\Ga_e,\ZZ) \big)$, and therefore the two kernels are identified by the contraction maps. We similarly identify $(\Coker \pi^*)^{\torf}$ and $(\Coker \pi_e^*)^{\torf}$, and it is clear that the pairings and polarizations agree. Hence we have $\Prym(\tGa/\Ga)\cong \Prym(\tGa_e/\Ga_e)$ and therefore $\Prym_p(\tGa/\Ga)\cong \Prym_p(\tGa_e/\Ga_e)$ as well.
\end{proof}  

By Lemma~\ref{lem:edgefree} it is possible to restrict attention to edge-free double covers, in other words replace each connected component of the dilation subgraph by a single dilated vertex. We now recall the notation for $\Ker p_*$ of a double cover $p:\tG\to G$ introduced in Sections~\ref{subsec:edgecontraction} and~\ref{subsec:Prym}. Recall that we chose a labeling $p^{-1}(v)=\{\tv^{\pm }\}$ of the undilated vertices $v\in V(G)\backslash V(G_{\dil})$. We then label $p^{-1}(h)=\{\th^{\pm}\}$ the preimages of the undilated half-edges, and require $r(\th^{\pm})=\widetilde{r(h)}^{\pm}$ if an undilated half-edge is rooted at an undilated vertex. The involution on $\tG$ is given by a parity assignment $\sigma:E_{\ud}(G)\to \{\pm 1\}$ via $\iota(\th^\pm)=\widetilde{\iota(h)}^{\pm\sigma(e)}$, where $\sigma(h)=\sigma \big(\iota(h) \big)=\sigma(e)$ for an edge $e=\big\{h,\iota(h) \big\}\in E(G)$, and we can further relabel to assume that $\sigma(e)=+1$ if $e$ has at least one dilated root vertex. Choosing consistent orientations for $\tG$ and $G$ determines a labeling $p^{-1}(e)=\{\te^+,\te^-\}$ for undilated edges $e\in E_{\ud}(G)$, and we represent elements $\gamma\in \Ker p_*$ as chains on the target graph (see Equation~\eqref{eq:chainsontarget}):
\[
\gamma=\sum_{e\in E(G)}\gamma(e)\cdot(\te^+-\te^-)= \sum_{e\in E(G)}\gamma(e)\cdot e\in C_1(G,\ZZ).
\]
A chain $\ga\in C_1(G,\ZZ)$ represents an element of $\Ker p_*$ if the above expression is closed as a chain on $\tG$. This is most conveniently stated in the language of oriented matroids, which we introduce in Section~\ref{subsec:orientedmatroids}.

We observed in Proposition~\ref{prop:signedgraphic} that a set $F\subset E(G)$ is dependent for the signed graphic matroid $M(\tG/G)$ if and only if the induced map $p[F]_*$ on homology has nontrivial kernel. If $F$ is a circuit, in other words a minimal dependent set, then the kernel is one-dimensional and has a unique generator up to sign. 

\begin{definition} Let $p:\tG\to G$ be an oriented double cover, let $C$ be a circuit of the signed graphic matroid $M(\tG/G)$, and let $p[C]:p^{-1}(G[C])\to G[C]$ be the restriction of $p$. The \emph{fundamental cycle $\gamma_C$ associated to }$C$
\[
\gamma_{C}=\sum_{e\in E(G)}\gamma_{C}(e)\cdot (\te^+-\te^-)=\sum_{e\in E(G)}\gamma_{C}(e)\cdot e
\]
is the unique (up to sign) generator of $\Ker p[C]_*$.
    \label{def:fundcycle}
\end{definition}

Figure~\ref{fig:circuits} shows the six circuit types and gives formulas for the fundamental cycles. Specifically, define the cycles $\ga'_C\in \Ker p_*$ by the formulas
\begin{equation}
    \gamma'_{C}(e)=\begin{cases} \pm 2, & \text{if } e\in C\mbox{ is a bridge of }G[C] ,\\
	\pm 1, & \text{if } e\in C\mbox{ is not a bridge of }G[C],\\
	0, & \text{if } e\notin C,
\end{cases}
\label{eq:fundcycle}
\end{equation}
where the signs are chosen in such a way that $\ga'_C$ is closed. Then it is easy to see that
\begin{equation}    
\ga_C=\frac{\ga'_C}{\gcd \big\{ \big|\ga'_C(e) \big|:e\in C \big\} }=
\begin{cases}
    \ga'_C, & \text{if } C\mbox{ is of types I-V},\\
    \ga'_C/2, & \text{if } C\mbox{ is of type VI}.
\end{cases}
\label{eq:fundcycle2}
\end{equation}


\subsection{The index function}\label{subsec:index} We now describe the additional decorations of the signed cographic matroid of a double cover that are needed to determine the coefficients of the cycles $\gamma_C$, and hence reconstruct $\Ker p_*$ and the Prym variety. 

Equation~\eqref{eq:fundcycle} tells us that we need a way to detect whether an edge $e$ of a circuit is a bridge or not. This information is contained in the matroid $M(G)$, but not in $M(\tG/G)$. Instead, we observe that Equation~\eqref{eq:Prymvolume} for the volume of the Prym (in contrast to Equation~\eqref{eq:Jacvolume} for the Jacobian) contains an extra quantity, namely the index of an ogod, which must therefore play a role in the reconstruction. We now extend this function from the ogods, which are the bases of $M^*(\tG/G)$, to all independent sets of $M^*(\tG/G)$.

\begin{definition} \label{def:index}
	For a double cover $p : \tilde G \to G$, we endow the matroid $M^*(\tilde G /G)$ with the \emph{index function} defined as 
	\[ \ind : \mathcal{I}\big(M^*(\tilde G / G)\big) \longrightarrow \ZZ_{>0}, \qquad F \longmapsto \text{number of connected components of } G \backslash F. \]
\end{definition}

We emphasize that the index function is defined graph-theoretically and cannot be computed from the matroid $M^*(\tG/G)$. In particular, when counting the connected components of $G \backslash F$, isolated dilated vertices are taken into account, so Zaslavsky's definition of $M^*(\tG/G)$ cannot be used to determine the index. We collect some basic properties of the index function in the following lemma, which we leave without proof.

\begin{lemma} \label{lemma:index}
    Let $p : \tilde G \to G$ be a double cover of graphs, let $M^* = M^*(\tilde G / G)$ be the signed cographic matroid, and let $\ind$ be the index function. Then the following hold.
    \begin{enumerate}
        \item If $F\in \mathcal{I}(M^*)$ and $F' \subset F$, then $\ind(F') \leq \ind(F)$.
        \item For every $F \in \mathcal{I}(M^*)$ we have $\ind(F) \leq \rk M^*+1$. If $F$ is a basis then we also have $d(\tG/G) \leq \ind(F)$.
        \item If $B, B' \in \mathcal{B}(M^*)$ differ by only one element, then $\big|\ind(B')-\ind(B) \big|\leq 1$.
    \end{enumerate}
\end{lemma}

Let $C$ be a circuit of the signed graphic matroid $M(\tG/G)$. We now show how to determine the absolute values of the coefficients of the fundamental cycle $\gamma_{C}\in \Ker p_*$. Let $e \in C$ be an edge. The set $C\backslash \{e\}$ is independent in $M(\tG/G)$ and is contained in a basis $E(G)\backslash F$, where $F\subset E(G)$ is an ogod. Each connected component of $G\backslash F$ contains a unique undilated cycle or a dilated component, but not both. Hence if $e\in C$ is a bridge, then the two connected components of $C\backslash \{e\}$ lie in different connected components of $G\backslash F$. In other words, $e$ is a bridge of $C$ if and only if $G\backslash F$ has one more connected component than $G\setminus \big(F\backslash \{e\} \big)$. Therefore, we have proved the following.

\begin{proposition} Let $p:\tG\to G$ be a double cover and let $C$ be a circuit of the signed graphic matroid $M(\tG/G)$. Then the absolute values of the coefficients in Equation~\eqref{eq:fundcycle} can be computed as
 \[
\big|\gamma'_C(e)\big|=\begin{cases}
    2^{\ind(F)-\ind(F\backslash\{e\})}, & \text{if } e\in C,\\
    0, & \text{else,}
\end{cases}
\]
where $F$ is any ogod such that $C\backslash \{e\}\subset E(G)\backslash F$. 
    \label{prop:indexformula}
\end{proposition}

\subsection{Oriented matroids} \label{subsec:orientedmatroids} To reconstruct the signs of the coefficients $\ga_C(e)$ of the fundamental cycles, it is convenient to use the language of oriented matroids. We recall the definitions (a standard reference is~ \cite{Oriented_matroids}). 



Given a finite set $E = \{e_1, \ldots, e_n\}$, a \emph{signed subset} $\overrightarrow{S}$ of $E$ is a subset of $E^+ \sqcup E^-$, where $E^+ = \{e_1^+, \ldots, e_n^+\}$ and $E^- = \{e_1^-, \ldots, e_n^-\}$ are copies of $E$. For a signed subset $\overrightarrow{S}$ of $E$ define 
\begin{itemize}
	\item the \emph{opposite set} $-\overrightarrow{S} = \big\{e^+ : e^- \in \overrightarrow{S} \big\} \cup \big\{e^- : e^+ \in \overrightarrow{S} \big\}$, 
	\item the \emph{underlying set} $S = \big\{e : e^+ \in \overrightarrow{S} \text{ or } e^- \in \overrightarrow{S} \big\}$, and
	\item the \emph{positive (resp. negative) part} $S^+ = \big\{ e: e^+ \in \overrightarrow{S} \big\}$ (resp. $S^- = \big\{ e: e^- \in \overrightarrow{S} \big\}$). 
\end{itemize}
It is understood that the underlying sets of $S^+$ and $S^-$ are disjoint. Therefore, for $e \in E$ we write
\[
\overrightarrow{S}(e)=\begin{cases}
    +1, & \text{if } e^+ \in \overrightarrow{S},\\
    -1, & \text{if } e^- \in \overrightarrow{S},\\
    0, &\mbox{otherwise}.
\end{cases}
\]

For our purposes, it is convenient to define an \emph{oriented matroid} $\overrightarrow{M}$ on the ground set $E$ in terms of its \emph{oriented circuits} $\mathcal{C}(\overrightarrow{M})$, a family of signed subsets of $E$ satisfying the following axioms:
\begin{enumerate}
	\item $\emptyset \not\in \mathcal{C}(\overrightarrow{M})$.
	\item If $\overrightarrow{S} \in \mathcal{C}(\overrightarrow{M})$ then $-\overrightarrow{S} \in \mathcal{C}(\overrightarrow{M})$. 
	\item For $\overrightarrow{S}, \overrightarrow{T} \in \mathcal{C}(\overrightarrow{M})$ with $S \subset T$ we have either $\overrightarrow{S} = \overrightarrow{T}$ or $\overrightarrow{S} = -\overrightarrow{T}$.
	\item For all $\overrightarrow{S}, \overrightarrow{T} \in \mathcal{C}(\overrightarrow{M})$ with $\overrightarrow{S} \neq \pm\overrightarrow{T}$ and $e \in S^+ \cap T^-$ there exists $\overrightarrow{U} \in \mathcal{C}(\overrightarrow{M})$ such that $U^+ \subset (S^+ \cup T^+) \setminus \{e\}$ and $U^- \subset (S^- \cup T^-) \setminus \{e\}$.
\end{enumerate}

For any oriented matroid $\overrightarrow{M}$ on $E$, there is an underlying matroid $M$ on $E$ whose circuits are the underlying sets $C$ of the oriented circuits $\overrightarrow{C} \in \mathcal{C}(\overrightarrow{M})$. Furthermore, for every circuit $C$ of $M$ there are precisely two circuits $\overrightarrow{C}_1, \overrightarrow{C}_2 \in \mathcal{C}(\overrightarrow{M})$ with underlying set $C$ and $\overrightarrow{C}_1 = - \overrightarrow{C}_2$.
A matroid $M$ is called \emph{orientable} if there is an oriented matroid $\overrightarrow{M}$ having underlying matroid $M$.


In~\cite{1991Zaslavsky}, Zaslavsky described how to orient the signed graphic matroid of a signed graph, and here we translate his construction into the language of double covers. To simplify exposition, we first describe how to orient the graphic matroid of an ordinary graph.

\begin{example} \label{ex:orientedgraphic}
    Let $G = \big(V(G), H(G), r,\iota \big)$ be a graph. An orientation on $G$ is defined by a function $o:H(G)\to \{\pm 1\}$ on the half-edges satisfying the property 
\begin{equation}
o(h)o \big(\iota(h) \big) =-1\label{eq:orient}
\end{equation}
for any edge $e=\{h,h'\}\in E(G)$. We view a half-edge $h\in H(G)$ rooted at $v=r(h)$ as pointing towards $v$ if $o(h)=+1$ and away from $v$ if $o(h)=-1$. With respect to this choice of orientation, we can define the simplicial chain complex of $G$ as
\[
d:\ZZ^{E(G)}\longrightarrow \ZZ^{V(G)},\qquad e\longmapsto \sum_{h\in e}o(h)\cdot r(h),
\]
and the simplicial homology group as
\[
H_1(G,\ZZ)=\Ker d=\left\{\sum_{e\in E(G)}\ga(e)\cdot e\in \ZZ^{E(G)}: \sum_{h\in T_vG} o(h)\ga(h)=0\mbox{ for all }v\in V(G)\right\},
\]
where $\ga(h)=\ga(e)$ for $h\in e$. Each circuit $C$ of $M(G)$ supports a primitive cycle $\delta_C \in H_1(G, \ZZ)$, unique up to sign, such that $\delta_C(e)=\pm 1$ for all $e \in C$ and $0$ otherwise. The circuits $\overrightarrow{C} = C^+ \sqcup C^-$ of the \emph{oriented graphic matroid} $\overrightarrow{M}(G)$ of the oriented graph $G$ have the form
    \[ C^+ = \big\{ e \in C : \delta_C(e) = +1 \big\} \qquad \text{and} \qquad C^- = \big\{ e \in C : \delta_C(e) = -1 \big\}, \]
    where $C$ ranges over the circuits of $M(G)$ and $\delta_C$ is one of the two primitive cycles supported on $C$. It is elementary to verify that $\overrightarrow{M}(G)$ is an oriented matroid on the ground set $E(G)$, having underlying matroid $M(G)$. The cycles $\delta_C$ can be recovered from $\overrightarrow{M}(G)$ up to sign and span $H_1(G,\ZZ)$.
    Hence the oriented matroid $\overrightarrow{M}(G)$ determines, after choosing a length function $\ell:E(G)\to \RR_{>0}$, the tropical Jacobian.

\end{example}

We now consider a double cover $p:\tG\to G$. Choose an orientation for $G$, in other words choose a function $o:H(G)\to \{\pm 1\}$ satisfying Equation~\eqref{eq:orient}. This induces an orientation on $\tG$, and in terms of the labeling introduced in Section~\ref{subsec:fundcycle}, any element $\ga\in \Ker p_*$ has the form
\[
\ga=\sum_{e\in E_{\ud}(G)}\ga(e)\cdot (\te^+-\te^-).
\]
Denoting $\ga(h)=\ga(e)$ when $h\in e$, it is then elementary to verify that 
\begin{equation}
\Ker p_*=\left\{
\sum_{e\in E_{\ud}(G)}\ga(e)\cdot (\te^+-\te^-):
\sum_{h\in T_vG} \tau(h)\ga(h)=0\mbox{ for all }v\in V(G)\backslash V(G_{\dil})
\right\}
\label{eq:Kerp}
\end{equation}
where the function $\tau:H_{\ud}(G)\to \{\pm 1\}$ is defined by 
\[
\tau(h)=\begin{cases}
    \sigma(e),& \text{if } o(h)=1\mbox{ and }h\in e,\\
    -1, & \text{if } o(h)=-1.
\end{cases}
\]
The function $\tau$ satisfies
\[
\tau(h)\tau \big(\iota(h) \big)=-\sigma(e)
\]
and defines an \emph{orientation of the signed graph} $(G,\sigma)$ (see Equation (2.2) in~\cite{1991Zaslavsky}). This in turn defines the structure of an oriented matroid $\overrightarrow{M}(\tG/G)$ on $M(\tG/G)$, and hence an induced orientation on $M^*(\tG/G)$, as follows (see Theorem 3.3 in~\cite{1991Zaslavsky}). Given an orientation $\tau:H_{\ud}(G)\to \{\pm 1\}$, Equation~\eqref{eq:Kerp} determines which $1$-chains supported on $E_{\ud}(G)$ represent elements of $\Ker p_*$. In particular, given a circuit $C\in \mathcal{C}\big(M(\tG/G)\big)$, we can now determine the signs of the coefficients of the fundamental cycle $\ga_C$, the absolute values having been determined in Section~\ref{subsec:index}:
\begin{equation}
	\ga_C=\sum_{e\in E_{\ud}(G)}\overrightarrow{C}(e) \big|\ga_C(e) \big|\cdot e=
	\sum_{e\in E_{\ud}(G)}\overrightarrow{C}(e) \big|\ga_C(e) \big|\cdot (\te^+-\te^-).
	\label{eq:cycleformula}
\end{equation}
The functions $\overrightarrow{C}(e)$ then represent the two oriented circuits of $\overrightarrow{M}(\tG/G)$ lying over $C$, for the two possible generators $\ga_C$ of $\Ker p[C]_*$.

\subsection{Reconstructing the Prym variety from the matroid} We now explain how to recover the Prym variety of a double cover $\pi:\tGa\to \Ga$ from the signed graphic matroid $M(\tGa/\Ga)$ equipped with the index function $\ind:\mathcal{I} \big(M^*(\tGa/\Ga) \big) \to \ZZ_{>0}$, orientation $\overrightarrow{M}(\tGa/\Ga)$, and edge length function $\ell:E(\Ga)\to \RR_{>0}$. The circuits $C\in \mathcal{C} \big(M(\tGa/\Ga) \big)$ determine fundamental cycles $\ga_C\in \Ker \pi_*$. Specifically, the index determines the $ \big|\ga_C(e) \big|$ by Equation~\eqref{eq:fundcycle2} and Proposition~\ref{prop:indexformula}, while the orientation determines the signs by Equation~\eqref{eq:cycleformula}. This data is sufficient to reconstruct $\Ker \pi_*$.

\begin{proposition} The lattice $\Ker \pi_*\subset H_1(\tGa,\ZZ)$ of a double cover $\pi:\tGa\to \Ga$ is spanned by the fundamental cycles $\ga_C$, where $C$ ranges over the circuits of $M(\tGa/\Ga)$.
\label{prop:fundcyclePrymspan}

\end{proposition}

\begin{proof} By definition, $\gamma_C \in \Ker \pi_*$ for all circuits $C$ of $M(\tGa / \Ga)$. Conversely, let $\pi:\tGa\to \Ga$ be a double cover. If the dilation index $d(\tGa/\Ga)=1$, then in fact $\Ker \pi_*$ has a basis consisting of fundamental cycles $\ga_C$ obtained by choosing an ogod $F\subset \mathcal{B}\big(M^*(\tGa/\Ga)\big)$ of minimal index; we prove this in Proposition~\ref{prop:fundcyclePrym}. Hence we assume that $d(\tGa/\Ga)>1$. Let $\ga\in \Ker \pi_*$. Choose an edge $e_1\in E_{\ud}(\Ga)$ lying on a type VI circuit $C_1$ (for example, choose any path connecting distinct connected components of $\Ga_{\dil}$). Then $\ga_{C_1}(e_1)=\pm 1$, and we choose $\ga_{C_1}$ so that $\ga_{C_1}(e_1)=1$. We then have $\ga=\ga(e_1)\ga_{C_1}+\ga_1$, where $\ga_1\in \Ker \pi_*$ and is supported on $\Ga\backslash \{e_1\}$.

We now remove $e_1$ from $\Ga$ and proceed by induction to obtain $\ga=\sum \ga(e_i)\ga_{C_i}+\ga'$, where each $\ga_{C_i}$ is a type VI circuit and $\ga'\in \Ker \pi_*$ is supported on the disconnected graph $\Ga'=\Ga\backslash \{e_1,\ldots,e_k\}$, which has no type VI circuits. It follows that each connected component of $\Ga'$ has connected dilation subgraph and hence dilation index one, so the restriction of $\ga'$ to each connected component lies in the span of the fundamental circuits by Proposition~\ref{prop:fundcyclePrym}. This completes the proof.
%
%
%
%
%
%
\end{proof}


For the reconstruction of $\Prym_p(\tGa/\Ga)$ it is convenient to collect all the necessary matroidal data in a definition.

\begin{definition} \label{def:PrymM} Consider the following data:
\begin{itemize}
    \item A signed graphic matroid $M$ on a ground set $E$, together with an orientation $\overrightarrow{M}$ induced by an oriented double cover representing $M$.
    \item An index function $\ind:\mathcal{I}(M^*)\to \ZZ_{>0}$ induced from the oriented double cover representing $\overrightarrow{M}$.
    \item An edge length function $\ell$ on $E$.
\end{itemize}
We define a pptav $\big(\Lambda, \Lambda, [\cdot, \cdot] \big)$, which depends on all of this data but which we simply denote $\Prym(M)$. For each circuit $C\in \mathcal{C}(M)$, let $\ga_C\in \ZZ^E$ be the fundamental cycle given by Equation~\eqref{eq:cycleformula}. We let $\Lambda$ be the lattice spanned in $\ZZ^E$ by the $\gamma_C$ for $C$ ranging over the oriented circuits of $M$, and define the pairing to be 
    \begin{equation}
        [\gamma_{C_1}, \gamma_{C_2}] = \sum_{e\in E} 2\gamma_{C_1}(e) \gamma_{C_2}(e) \ell(e).
        \label{eq:Prympairing}
    \end{equation}
Finally, the polarization $\xi:\La\to \La$ is the identity map. 
\end{definition}

We now state the main result of our paper. 

\begin{theorem} 
\label{thm:Prym_from_matroid}
	Let $\pi : \tilde \Gamma \to \Gamma$ be a double cover of metric graphs. Choose an orientation for $\pi$, and let $M = M(\tilde\Gamma / \Gamma)$ be the signed graphic matroid of $\pi$ with induced orientation $\overrightarrow{M}$. Let $\ind:\mathcal{I}(M^*)\to \ZZ_{>0}$ be the index function and let $\ell:E(\Ga)\to \RR_{>0}$ be the edge length function. Then $\Prym(M)$ is a pptav and
	\[ \Prym_p(\tilde \Gamma / \Gamma) \cong \Prym(M) \]
    as pptavs. In particular, $\Prym(M)$ only depends on the reorientation class of $\overrightarrow{M}$ and not on $\overrightarrow{M}$ itself.
\end{theorem}
\begin{proof} Recall that the Prym variety $\Prym(\tGa/\Ga)$ is the polarized tropical abelian variety defined by the lattices $K=(\Coker \pi^*)^{\torf}$, $K'=\Ker \pi_*$, and the pairing $[\cdot,\cdot]$ and polarization $\xi:K'\to K$ induced by the integration pairing and polarization on $\tGa$. The principalization $\Prym_p(\tGa/\Ga)$ is obtained by replacing $K$ by the image $\Im \xi$. By Proposition~\ref{prop:fundcyclePrymspan}, $K'=\Im \xi=\Lambda$ is the lattice spanned by the fundamental cycles $\ga_C$ corresponding to the oriented circuits of $M$. To verify that~\eqref{eq:Prympairing} is the pairing on $\Prym_p(\tGa/\Ga)$ given in Equation~\eqref{eq:pairingPrym}, we simply note that
\[
[\te_1^+-\te_1^-,\te_2^+-\te_2^-]=\begin{cases} 2\ell(e_1),& \text{if } e_1=e_2,\\
0, & \text{if } e_1\neq e_2
\end{cases}
\]
for any two undilated edges $e_1,e_2\in E_{\ud}(\Ga)$.
In particular, $\Prym(M)$ is a pptav and does not change under reorientation of $M$, which is equivalent to reorienting the cover $\tGa \to \Ga$.
\end{proof}

\begin{remark} \label{rem:reconstruction}
    We note that we are not currently able to reconstruct, directly from the matroid, the Prym variety $\Prym(\tGa/\Ga)$, except in the following two cases. Recall that the induced polarization on $\Prym(\tGa/\Ga)$ has type $(1,\ldots,1,2,\ldots,2)$, where the number of $1$s is equal to $d(\tGa/\Ga)-1$. If $d(\tGa/\Ga)=1$ (in other words, if $\pi:\tGa\to \Ga$ is free or has connected dilation subgraph), then the induced polarization is twice a principal polarization, and we can reconstruct $K$ by doubling the lattice $K'=\Ker \pi_*$. On the other hand, if $d(\tGa/\Ga)=g(\tGa)-g(\Ga)+1$ (equivalently, if $\Ga$ is a tree, as in Example~\ref{ex:trees}), then the induced polarization on $\Prym(\tGa/\Ga)$ is already principal and $\Prym_p(\tGa/\Ga)=\Prym(\tGa/\Ga)$. We also note that the dilation index $d(\tGa/\Ga)$ (and hence the polarization type) can in fact be read off the indexed matroid $M(\tGa/\Ga)$ as the minimal index of an ogod $F\in \mathcal{B}\big(M^*(\tGa/\Ga) \big)$ (see Lemma~\ref{lemma:index} and Lemma~\ref{lemma:ogodindex}). 
\end{remark}

The following is an immediate consequence of Theorem~\ref{thm:Prym_from_matroid}.

\begin{corollary}  \label{cor:criterion}
    Let $\pi_i : \tilde\Gamma_i \to \Gamma_i$ for $i = 1,2$ be double covers. Let $(M_i, \overrightarrow{M}_i, \ind_i, \ell_i)$ be the associated matroidal data packages, i.e. $M_i = M(\tilde\Gamma_i / \Gamma_i)$ with oriented matroid structure $\overrightarrow{M}_i$ induced by the choice of some orientation on $\pi_i$, $\ind_i$ is the index function on $M^*_i$ and $\ell_i$ the edge-length function on $M_i$. If there is an isomorphism of the matroidal data packages, i.e. a bijection $E_\ud(\Gamma_1) \to E_\ud(\Gamma_2)$ commuting with the $\ell_i$ which induces a bijection $\mathcal{I}(M^*_1) \to \mathcal{I}(M^*_2)$ commuting with the $\ind_i$ and which induces an isomorphism between $\overrightarrow{M}_1$ and a reorientation of $\overrightarrow{M}_2$    
    then $\Prym_p(\tilde \Gamma_1 / \Gamma_1) \cong \Prym_p(\tilde \Gamma_2 / \Gamma_2)$ as pptavs.
\end{corollary}

A natural question is to what extent a double cover can be recovered from the associated matroid (cf. Problem~5.1 in \cite{1982Zaslavsky}). We give an example, and explore this question from a matroid-theoretic perspective in Section~\ref{sec:fibers}.

\begin{example} \label{ex:trees}
    Consider a double cover $\pi : \tilde \Gamma \to \Gamma$ where $\Gamma$ is a tree consisting of $n$ undilated edges all of length 1, and such that every vertex is dilated. We claim that all such covers have the same Prym variety. Indeed, since all vertices are dilated, removing edges never relatively disconnects $\pi$. Therefore the signed cographic matroid $M^*(\tGa / \Ga) = U_{n, n}$ is uniform, and the index function is $\ind(F) = |F|+1$, neither of which depend on the tree $\Gamma$. Moreover, since $U_{n, n}$ is in fact a graphic matroid, it is uniquely orientable up to reorientation. This shows that up to reorientation the matroidal package of $\pi$ depends only on $n$, and we conclude from Corollary~\ref{cor:criterion} that any two double covers $\pi_1$ and $\pi_2$ of the aforementioned form have the same Prym variety.
\end{example}

\subsection{Effective reconstruction and a basis of fundamental cycles}
In Theorem~\ref{thm:Prym_from_matroid} we reconstruct the principalized tropical Prym variety from the matroidal data of the double cover, by considering the fundamental cycles corresponding to all circuits of $M(\tGa/\Ga)$. This procedure is not computationally effective. In the case of graphs, to reconstruct the Jacobian, it is sufficient to use the circuits associated to a single spanning tree, obtained by adding each of the complementary edges one by one. We now explore whether an analogous construction produces the Prym variety. In Lemma~\ref{lemma:index}, we noted that the index function on a basis of $M^*(\tG/G)$ is bounded from below by the dilation index $d(\tG/G)$. We now show that this bound is sharp.

\begin{lemma} Let $p:\tG\to G$ be a double cover with $d(\tG/G)=1$ (so $p$ is either free or has connected dilation subgraph). Then there exists an ogod $F\subset E(G)$ of index one.
    \label{lemma:ogodindex}
\end{lemma}

\begin{proof} If $p$ is free, we pick an odd cycle on $G$, having connected preimage in $\tG$. We then construct $F$ by iteratively removing edges of $G$ without disconnecting $G$ or breaking the chosen cycle. If $p$ is dilated with connected dilation subgraph, we assume for simplicity that $p$ is edge-free with a unique dilated vertex. Then $F$ is the set of edges in the complement of any spanning tree.
\end{proof}

We now show that a basis for $\Ker p_*$ can be constructed by starting with an ogod of index one. This gives an exact analogue for Pryms of the fundamental cycle construction for Jacobians. 

\begin{proposition} Let $p:\tG\to G$ be a double cover, and suppose that $F=\{e_1,\ldots,e_h\}\subset E(G)$ is an ogod of index one. For each $i=1,\ldots,h$, let $C_i$ denote the unique circuit of $M(\tG/G)$ in $G\backslash F\cup \{e_i\}$. Then the fundamental cycles $\ga_{C_i}$ are a basis for $\Ker p_*$ .

\label{prop:fundcyclePrym}
\end{proposition}

\begin{proof} By Lemma~\ref{lemma:index} the index function is monotonous, so $\ind(F)=1$ implies that $\ind(F\backslash\{e_i\})=1$ for all $i=1,\ldots,h$. Hence Proposition~\ref{prop:indexformula} shows that $\ga_{C_i}(e_i)=\pm 1$, and we choose signs so that $\ga_{C_i}(e_i)=1$. In addition, $e_j\notin C_i$ and therefore $\ga_{C_i}(e_j)=0$ for $j\neq i$. 

Now let $\ga\in \Ker p_*$. Since $F$ is an ogod, the map $p^{-1}(G\backslash F)\to G\backslash F$ is injective on $H_1$. Hence $\ga=0$ if $\ga(e_i)=0$ for all $i$. It follows that
\[
\ga=\ga(e_1)\ga_{C_1}+\cdots+\ga(e_h)\ga_{C_h},
\]
and therefore the $\ga_{C_i}$ are a basis.
\end{proof}

The argument above shows that the fundamental cycles $\ga_{C_i}$ associated to an ogod $F$ of index $\ind(F)\geq 2$ are linearly independent and span a finite index sublattice of $\Ker p_*$. The index of this lattice in $\Ker p_*$ does not seem to be directly related to $\ind(F)$. For example, the unique ogod of Example~\ref{ex:trees} has index $n$ but gives a basis for $\Ker p_*$. On the other hand, Figure~\ref{fig:no_ogod} gives an example of a double cover with dilation index two for which no ogod gives a basis for $\Ker p_*$. Hence the fundamental cycle construction using a single ogod cannot generally be used when the dilation subgraph is disconnected.

    
    \begin{figure}
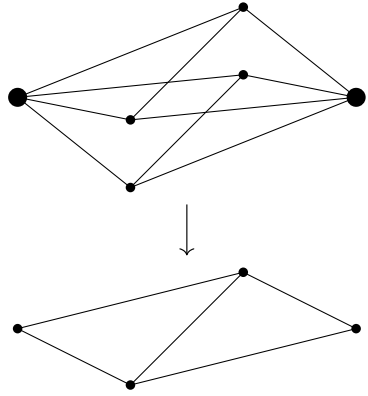

        \centering
        \doublecover{
            \vertex[2]{0}{0}
            \vertex[2]{3}{0}

            \foreach \i  in {-0.3, 0.3}{
                \vertex{1}{-0.5 + \i}
                \vertex{2}{0.5 + \i}
            
                \path[draw] (0,0) -- (1, -0.5 + \i) -- (2, 0.5 + \i) -- (3, 0);
                \path[draw] (0,0) -- (2, 0.5 + \i);
                \path[draw] (1, -0.5 + \i) -- (3,0);
            }
            
        }{
            \vertex{0}{0}
            \vertex{1}{-0.5}
            \vertex{2}{0.5}

            \path[draw] (0,0) -- (1, -0.5) -- (2, 0.5) -- (3, 0);
            \path[draw] (0,0) -- (2, 0.5);
            \path[draw] (1, -0.5) -- (3,0);

            \vertex{3}{0}
        
        }        
        \caption{Double cover with no ogod producing a basis for $\Ker p_*$}
        \label{fig:no_ogod}
    \end{figure}

\section{Invariance of \texorpdfstring{$\Prym_p(\tGa/\Ga)$}{} under the simplification of \texorpdfstring{$M^*(\tGa/\Ga)$}{}}
\label{sec:fibers}

The papers~\cite{2010CaporasoViviani} and~\cite{2011BrannettiMeloViviani} considered the tropical Torelli problem, in other words whether and to what extent a metric graph $\Ga$ can be reconstructed from its Jacobian $\Jac(\Ga)$. The authors reconstruct $\Jac(\Ga)$ from the graphic matroid $M(\Ga)$ and show that the Jacobian does not change when $M(\Ga)$ is replaced by its simplification, and that conversely the simplification class of the matroid determines the Jacobian.  

In this section, we construct the simplification of the signed cographic matroid $M^*(\tGa/\Ga)$ of a double cover of metric graphs $\pi:\tGa\to \Ga$, and show that $\Prym_p(\tGa/\Ga)$ does not change under simplification. We then give an example of two double covers having non-isomorphic simple matroids but isomorphic Pryms, contrary to the case of Jacobians. Hence the fibers of the tropical Prym--Torelli map are not fully described by the simplification procedure.


\subsection{Simplification of the signed cographic matroid} We recall that the simplification $M_{\simp}$ of a matroid $M$ is given by deleting all $1$-circuits, and then iteratively deleting elements in 2-circuits until there are no more 2-circuits. For the cographic matroid $M^*(\Ga)$, the simplification $M_{\simp}^*(\Ga)$ is obtained by contracting all bridges in $\Gamma$, and an edge $e_1$ in each pair $\{e_1,e_2\}$ of separating edges. It is shown in~\cite{2010CaporasoViviani} that if we also set the new length of $e_2$ to be $\ell(e_1)+\ell(e_2)$, then $\Jac(\Ga)$ does not change.

We now adapt this procedure to the signed cographic matroid $M^*(\tGa/\Ga)$ of a double cover $\pi:\tGa\to \Ga$ of metric graphs. A circuit of $M^*(\tilde \Gamma / \Gamma)$ is a minimal set of undilated edges $F$ of $\Gamma$ such that the restricted cover $\pi^{-1}(\Ga\backslash F) \to \Gamma\backslash F$ is not relatively connected. The types of the 1- and 2-circuits of $M^*(\tGa/\Ga)$ are shown on Figures~\ref{fig:1circuits} and~\ref{fig:2circuits}. We label the edges of a $2$-circuit with multiplicities as follows.


\begin{definition} \label{def:simplification_mult} Let $f_i$ be an edge lying in a $2$-circuit $F = \{f_1, f_2\}$ of the signed cographic matroid $M^*(\tGa/\Ga)$. The \emph{multiplicity of $f_i$} is 
    \[ \mult(f_i) = 2^{\ind(\{f_i\}) - \ind(\emptyset)} = 2^{\ind(\{f_i\}) - 1} \in \{1, 2\}. \]
\end{definition}
We note that $\mult(f)=2$ if and only if $f$ is a bridge edge of $\Ga$. We now define a non-canonical simplification process for double covers of metric graphs.

\begin{definition} \label{def:simplification}
	Let $\pi : \tilde\Gamma \to \Gamma$ be a double cover of metric graphs. The \emph{simplification} $\pi_\simp : \tilde\Gamma_\simp \to \Gamma_\simp$ of $\pi$ is defined as follows. First, we contract all edges of $\Gamma$ that are $1$-circuits of $M^*(\tilde\Gamma /\Gamma)$ (see Figure~\ref{fig:1circuits}). Then, for every $2$-circuit $\{f_1, f_2\}$ of $M^*(\tilde\Gamma / \Gamma)$ (see Figure~\ref{fig:2circuits}), we contract one of the edges, say $f_1$, and redefine the length of the other edge to be $\ell(f_2) + \left( \frac{\mult(f_1)}{\mult(f_2)} \right)^2 \ell(f_1)$. The contracted edge can be chosen arbitrarily except for a $2$-circuit $\{f_1,f_2\}$ with $\mult(f_1)=2$ and $\mult(f_2)=1$ (lower left diagram on Figure~\ref{fig:2circuits}), for which we always contract the edge $f_1$ having multiplicity $2$. 
 
\end{definition}

By construction, $M^*(\tGa_{\simp}/\Ga_{\simp})$ is the simplification of $M^*(\tGa/\Ga)$. We note that we do not contract a loop whose preimage is a pair of parallel edges, except for the $2$-circuit on the top right of Figure~\ref{fig:2circuits}. This type of circuit only occurs for free double covers, and the resulting contraction is a double cover with a single dilated vertex. Therefore, simplification does not change the dilation index of a double cover. We now show that the Prym variety does not change under simplification.

\begin{theorem} \label{thm:simplification}
	Let $\pi : \tilde\Gamma \to \Gamma$ be a double cover of metric graphs and let $\pi_\simp : \tilde\Gamma_\simp \to \Gamma_\simp$ be a simplification of $\pi$. Then $\Prym
 _p(\tilde\Gamma / \Gamma) \cong \Prym_p(\tilde\Gamma_\simp / \Gamma_\simp)$. 
\end{theorem}

\begin{proof}
    Choose an orientation for the double cover $\pi$. This induces oriented matroid structures on $M(\tilde\Gamma / \Gamma)$ and $M^*(\tilde\Gamma / \Gamma)$. By \cite[Theorem~3.4.3]{Oriented_matroids} for every oriented circuit $\overrightarrow{C}\in \mathcal{C} \big(\overrightarrow{M}(\tGa/\Ga) \big)$ and every oriented circuit $\overrightarrow{D}\in\mathcal{C} \big(\overrightarrow{M}^*(\tGa/\Ga) \big)$, either $C \cap D = \emptyset$ or there are distinct edges $e,f \in C \cap D$ such that $\overrightarrow{C}(e)\overrightarrow{D}(e) = -\overrightarrow{C}(f)\overrightarrow{D}(f)$. This implies that if  $D=\{f\}$ is a 1-circuit of $M^*(\tGa/\Ga)$, then $f$ does not lie in any circuit of $M(\tilde\Gamma / \Gamma)$. But the Prym variety is generated by the fundamental cycles of the circuits of $M(\tilde\Gamma / \Gamma)$, and thus contracting $f$ does not change the Prym variety.

    Now let $D=\{f_1,f_2\}$ be a 2-circuit of $M^*(\tGa/\Ga)$ and let $C_1, C_2$ be any circuits of $M(\tGa/\Ga)$, possibly $C_1 = C_2$. We show that the pairing~\eqref{eq:pairingPrym} 
    \[
    [\ga_{C_1},\ga_{C_2}]=\sum_{e\in E(\Gamma)} 2 \ga_{C_1}(e)\ga_{C_2}(e)\ell(e) = \sum_{e\in C_1\cap C_2} 2 \big|\ga_{C_1}(e) \big| \big|\ga_{C_2}(e) \big|\overrightarrow{C}_1(e)\overrightarrow{C}_2(e)\ell(e)
    \]
    does not change under the suggested contraction of $f_1$. If $C_1 \cap C_2 \cap D = \emptyset$ then neither $f_1$ nor $f_2$ contribute to $[\ga_{C_1},\ga_{C_2}]$, so we assume without loss of generality that $f_1 \in C_1 \cap C_2$, and hence $f_2 \in C_1 \cap C_2$ as well by the aforementioned theorem. Still without loss of generality we may assume that $\overrightarrow{C}_1(f_1) = \overrightarrow{C}_2(f_1)$ (if not, simply replace $\overrightarrow{C}_1$ with $-\overrightarrow{C}_1$) and hence $\overrightarrow{C}_1(f_1)\overrightarrow{C}_2(f_1)=1$. It follows that 
    \[ -\overrightarrow{C}_1(f_2)\overrightarrow{D}(f_2) = \overrightarrow{C}_1(f_1)\overrightarrow{D}(f_1) = \overrightarrow{C}_2(f_1)\overrightarrow{D}(f_1) = -\overrightarrow{C}_2(f_2)\overrightarrow{D}(f_2), \]
    and hence$\overrightarrow{C}_1(f_2) = \overrightarrow{C}_2(f_2)$ and $\overrightarrow{C}_1(f_2)\overrightarrow{C}_2(f_2)=1$ as well. Therefore, the combined contribution from $f_1$ and $f_2$ to $[\ga_{C_1},\ga_{C_2}]$ is equal to 
    \begin{equation} 
    2\big|\ga_{C_1}(f_1)\big| \big|\ga_{C_2}(f_1)\big|\ell(f_1)+    2\big|\ga_{C_1}(f_2)\big| \big|\ga_{C_2}(f_2)\big|\ell(f_2).
    \label{eq:beforecontraction}
    \end{equation}
    
We now contract $f_1$ and rescale the length of $f_2$ to $\ell'(f_2)=\ell(f_2) + \left( \frac{\mult(f_1)}{\mult(f_2)} \right)^2 \ell(f_1)$. Denote $C'_1=C_1\backslash \{f_1\}$ and $C'_2=C_2\backslash \{f_1\}$ 
the contractions of $C_1$ and $C_2$, respectively. Orientations are preserved under contraction, hence $\overrightarrow{C}'_1(f_2)\overrightarrow{C}'_2(f_2)=\overrightarrow{C}_1(f_2)\overrightarrow{C}_2(f_2)=1$ and the contribution from $f_2$ to $[\ga_{C'_1},\ga_{C'_2}]$ is equal to 
\begin{equation}
2\big|\ga_{C'_1}(f_2)\big| \big|\ga_{C'_2}(f_2)\big| \ell'(f_2) =
2\big|\ga_{C'_1}(f_2)\big| \big|\ga_{C'_2}(f_2) \big| \left( \frac{\mult(f_1)}{\mult(f_2)} \right)^2 \ell(f_1) +
2\big|\ga_{C'_1}(f_2)\big| \big|\ga_{C'_2}(f_2)\big| \ell(f_2).
\label{eq:aftercontraction}
\end{equation}
To show that the right hand side of~\eqref{eq:aftercontraction} is equal to~\eqref{eq:beforecontraction}, it is enough to check for any circuit $C\in M(\tGa/\Ga)$ containing $D=\{f_1,f_2\}$ and contracting to $C'=C\backslash \{f_1\}$, we have
\[
\big|\ga_C(f_2)\big| = \big|\ga_{C'}(f_2) \big|\qquad \text{and} \qquad \big|\ga_{C}(f_1) \big| =
\big|\ga_{C'}(f_2) \big|\frac{\mult(f_1)}{\mult(f_2)}.
\]
This can be verified on a case-by-case basis by comparing types of $2$-circuits of $M^*(\tGa/\Ga)$ (Figure~\ref{fig:2circuits}) and circuits of $M(\tGa/\Ga)$ (Figure~\ref{fig:circuits}). Hence the Prym variety does not change when $f_1$ is contracted. Proceeding in this manner, we see that $\Prym_p(\tGa_{\simp}/\Ga_{\simp})$ is isomorphic to $\Prym_p(\tGa/\Ga)$.
%
%
%
%
%
\end{proof}

\subsection{The Prym does not determine the simplified matroid} \label{subsec:isomPryms} We now show that, unlike the case of graphs, double covers with non-isomorphic simple matroids may still have isomorphic Prym varieties.

	\begin{figure}
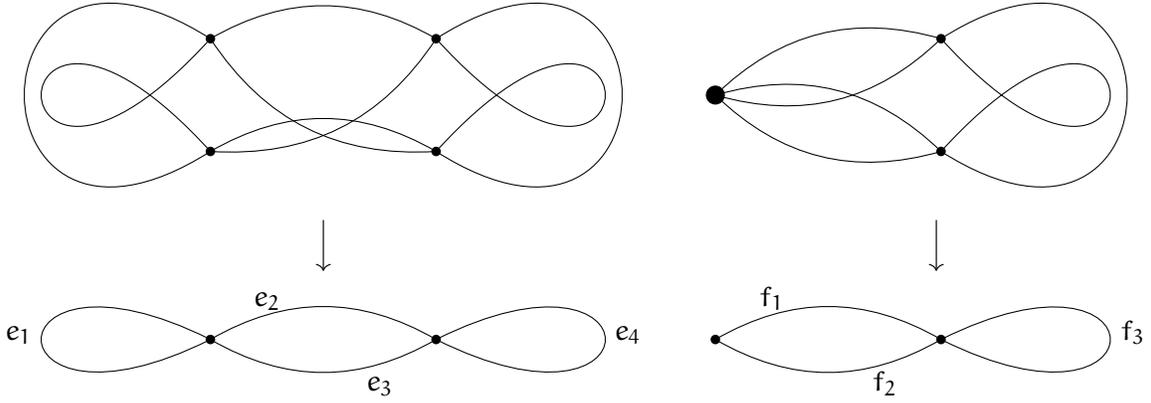

		\centering
		\doublecover{
			\clip (-2,-0.5) rectangle (4, 1.5);
			
			\vertex{0}{0}
			\vertex{0}{1}
			\vertex{2}{0}
			\vertex{2}{1}
			
			\spiralleft{0}{0}
			\path[draw] (0,0) to[bend left = 30] (2,0);
			\path[draw] (0,1) to[bend left = 30] (2,1);
			\path[draw] (0,0) to[bend right = 30] (2,1);
			\path[draw] (0,1) to[bend right = 30] (2,0);
			\spiral{2}{0}
			
		}{
			\clip (-2, -0.5) rectangle (4, 0.5);
			\vertex{0}{0}
			\vertex{2}{0}
			\drawselfloop{2}{0}
			\drawselfloopleft{0}{0}
			\path[draw] (0,0) to[bend left = 30] (2,0);
			\path[draw] (0,0) to[bend right = 30] (2,0);
			\draw (-1.7, 0) node {$e_1$};
			\draw (0.5, 0.3) node {$e_2$};
			\draw (1.5, -0.45) node {$e_3$};
			\draw (3.7, 0) node {$e_4$};
		}
		\doublecover{
			\clip (-2.4pt,-0.5) rectangle (4, 1.5);
			
			\vertex[2]{0}{0.5}
			\vertex{2}{0}
			\vertex{2}{1}
			
			\path[draw] (0,0.5) to[bend left = 30] (2,0);
			\path[draw] (0,0.5) to[bend left = 30] (2,1);
			\path[draw] (0,0.5) to[bend right = 30] (2,1);
			\path[draw] (0,0.5) to[bend right = 30] (2,0);
			\spiral{2}{0}
			
		}{
			\clip (-2.4pt, -0.5) rectangle (4, 0.5);
			\vertex{0}{0}
			\vertex{2}{0}
			\drawselfloop{2}{0}
			\path[draw] (0,0) to[bend left = 30] (2,0);
			\path[draw] (0,0) to[bend right = 30] (2,0);
			\draw (0.5, 0.3) node {$f_1$};
			\draw (1.5, -0.45) node {$f_2$};
			\draw (3.7, 0) node {$f_3$};
		}
		\caption{Double covers with isomorphic Pryms and non-isomorphic simple matroids.}
		\label{fig:not_necessary}
	\end{figure}

	Consider the double cover $\pi : \tilde \Gamma \to \Gamma$ depicted on Figure~\ref{fig:not_necessary} on the left, with edges $E(\Gamma) = \{e_1, e_2, e_3, e_4\}$ of lengths $x_i=\ell(e_i)$. The signed graphic matroid  $M(\tilde\Gamma / \Gamma) \cong U_{2,4}$ is the uniform matroid of rank 2 on 4 elements and its dual $M^*(\tilde\Gamma / \Gamma) = U_{2,4}^* = U_{2,4}$ is simple. To compute the Prym variety, we apply Proposition~\ref{prop:fundcyclePrym} to the index-1 ogod $\{e_3,e_4\}$. We obtain the following basis for $\Ker \pi_*$:
	\[ \gamma_1 = e_1 + e_2 + e_3, \qquad \gamma_2 = e_1 + 2e_2 + e_4 \]
    Thus the pairing $[\cdot, \cdot]$ on $\Prym_p(\tGa/\Ga)$ has Gram matrix 
	\begin{equation*}
		\begin{pmatrix}
			x_1 + x_2 + x_3 &  x_1 + 2x_2 \\
			x_1 + 2x_2 & x_1 + 4x_2 + x_4
		\end{pmatrix}.
	\end{equation*}

	Now we study the double cover $\sigma : \tilde\Delta \to \Delta$ on the right of Figure~\ref{fig:not_necessary} with edges $E(\Delta) = \{f_1, f_2, f_3\}$ of lengths $\ell(f_j)=y_j$. This time the signed graphic matroid is $M(\tilde\Delta / \Delta) = U_{1,3}$ and again the dual matroid $M^*(\tilde\Delta / \Delta) = U_{1,3}^* = U_{2,3}$ is simple, but obviously $U_{2,3} \not\cong U_{2,4}$. The ogod $\{f_2,f_3\}$ produces the basis 
	\[ \delta_1 = f_1 + f_2, \qquad \delta_2 = 2f_1 + f_3 \]
	for $\Ker \sigma_*$, and the Gram matrix of the pairing is
	\begin{equation*}
		\begin{pmatrix}
			y_1 + y_2 & 2y_1 \\
			2y_1 & 4y_1 + y_3
		\end{pmatrix}.
	\end{equation*}
 It is elementary to verify that the two Gram matrices are equal, and hence $\Prym_p(\tGa/\Ga)\cong \Prym_p (\widetilde\Delta/\Delta)$, for appropriate positive values of $x_i$ and $y_i$. 

	\bibliographystyle{alpha}
	\bibliography{references}
\end{document}

%% file: free_resolution.tex
\doublecover{
	\vertex[2]{0}{0}
	\vertex[2]{1}{0}
	\path[draw, line width = 1.5pt] (0,0) to[bend left = 30] (1,0) to[bend left = 30] (0,0);
	
	\foreach \i in {130, 140, 175, 185, 220, 230}{
		\path[draw] (0,0) --+ (\i:0.5);
	}

	\foreach \i in {30, 40, -30, -40}{
		\path[draw] (1,0) --+ (\i:0.5);
	}
}{
	\vertex{0}{0}
	\vertex{1}{0}
	\path[draw] (0,0) to[bend left = 30] (1,0) to[bend left = 30] (0,0);
	\path[draw] (0, 0) --+ (140:0.5);
	\path[draw] (0, 0) --+ (180:0.5);
	\path[draw] (0, 0) --+ (220:0.5);
	\path[draw] (1, 0) --+ (35:0.5);
	\path[draw] (1, 0) --+ (-35:0.5);
}%
\qquad $\rightsquigarrow$ \qquad
\doublecover{
	\vertex[2]{0}{0}
	
	\foreach \i in {130, 140, 175, 185, 220, 230}{
		\path[draw] (0,0) --+ (\i:0.5);
	}
	
	\foreach \i in {30, 40, -30, -40}{
		\path[draw] (0,0) --+ (\i:0.5);
	}
}{
	\vertex{0}{0}
	\path[draw] (0, 0) --+ (140:0.5);
	\path[draw] (0, 0) --+ (180:0.5);
	\path[draw] (0, 0) --+ (220:0.5);
	\path[draw] (0, 0) --+ (35:0.5);
	\path[draw] (0, 0) --+ (-35:0.5);
}%
\qquad $\leftsquigarrow$ \qquad
\doublecover{
	\vertex{0}{0}
	\vertex{0}{1}
	\path[draw] (0,0) to[bend left = 30] (0, 1) to[bend left = 30] (0,0);
	
	\foreach \i in {140, 180, 220, 35, -35}{
		\path[draw] (0,0) --+ (\i:0.5);
		\path[draw] (0,1) --+ (\i:0.5);
	}
}{
	\clip(-0.5, -0.8) rectangle (0.5, 0.4);
	\vertex{0}{0}
	\path[draw] (0,0) .. controls (0.5, -1) and (-0.5, -1) .. (0,0);
	\path[draw] (0, 0) --+ (140:0.5);
	\path[draw] (0, 0) --+ (180:0.5);
	\path[draw] (0, 0) --+ (220:0.5);
	\path[draw] (0, 0) --+ (35:0.5);
	\path[draw] (0, 0) --+ (-35:0.5);
}

%% file: elementary_double_covers.tex
\doublecover{
	\foreach \i in {-0.07, 0.07}{			
		\vertex{-2}{0.5 + \i}
		\vertex{-2}{-0.25 + \i}
		\vertex{-2}{-0.75 + \i}
		\vertex{-1}{0.5 + \i}
		\vertex{-1}{-0.5 + \i}
		\vertex{2}{0.5 + \i}
		\vertex{2}{0 + \i}
		\vertex{2}{-0.5 + \i}
		\vertex{3}{0.25 + \i}
		\vertex{3}{-0.25 + \i}
		
		\path[draw] (-2, 0.5 + \i) -- (-1, 0.5 + \i) -- (0, 0) -- (-1, -0.5 + \i) -- (-2, -0.25 + \i);
		\path[draw] (-1, -0.5 + \i) -- (-2, -0.75 + \i);
		\path[draw] (1, 0) -- (2, 0.5 + \i);
		\path[draw] (1, 0) -- (2, 0 + \i) -- (3, 0.25 + \i);
		\path[draw] (1, 0) -- (2, -0.5 + \i);
		\path[draw] (2, 0 + \i) -- (3, -0.25 + \i);	
	}
    \vertex[2]{0}{0}
    \vertex[2]{1}{0}
    \path[draw, line width = 1.5pt] (0, 0) to[bend left = 30] (1,0);
    \path[draw, line width = 1.5pt] (0, 0) to[bend right = 30] (1,0);
}{
	\vertex{-2}{0.5}
	\vertex{-2}{-0.25}
	\vertex{-2}{-0.75}
	\vertex{-1}{0.5}
	\vertex{-1}{-0.5}
	\vertex{0}{0}
	\vertex{1}{0}
	\vertex{2}{0.5}
	\vertex{2}{0}
	\vertex{2}{-0.5}
	\vertex{3}{0.25}
	\vertex{3}{-0.25}
	
	\path[draw] (-2, 0.5) -- (-1, 0.5) -- (0, 0) -- (-1, -0.5) -- (-2, -0.25);
	\path[draw] (-1, -0.5) -- (-2, -0.75);
	\path[draw] (0, 0) to[bend left = 30] (1, 0) -- (2, 0.5);
	\path[draw] (0, 0) to[bend right = 30] (1, 0) -- (2, 0) -- (3, 0.25);
	\path[draw] (1, 0) -- (2, -0.5);
	\path[draw] (2, 0) -- (3, -0.25);
}%
\hspace{1cm}%
\doublecover{
	\foreach \i in {-0.07, 0.07}{
		\vertex{-2}{0.5 + \i}
		\vertex{-2}{0 + \i}
		\vertex{-2}{-0.5 + \i}
		\vertex{-1}{0 + \i}
		\vertex{-0.5}{0.5 + \i}
		\vertex{0}{0 + \i}
		\vertex{0.5}{0.5 + \i}
		\vertex{1}{0 + \i}
		\vertex{1.5}{0.5 + \i}
		\vertex{2}{0 + \i}
		\vertex{2}{-0.5 + \i}
		\vertex{2.5}{0.25 + \i}
		\vertex{2.5}{0.75 + \i}
		
		\path[draw] (-2,0 + \i) -- (0, 0 + \i);
		\path[draw] (1, 0 + \i) -- (2, 0 + \i);
		\path[draw] (-2, 0.5 + \i) -- (-1, 0 + \i) -- (-2, -0.5 + \i);
		\path[draw] (0, 0 + \i) -- (0.5, 0.5 + \i) -- (1, 0 + \i) -- (2, -0.5 + \i);
		\path[draw] (-0.5, 0.5 + \i) -- (1.5, 0.5 + \i) -- (2.5, 0.75 + \i);
		\path[draw] (1.5, 0.5 + \i) -- (2.5, 0.25 + \i);
	}
	\path[draw] (0, -0.07) -- (1, 0.07);
	\path[draw] (0, 0.07) -- (1, -0.07);
}{
	\vertex{-2}{0.5}
	\vertex{-2}{0}
	\vertex{-2}{-0.5}
	\vertex{-1}{0}
	\vertex{-0.5}{0.5}
	\vertex{0}{0}
	\vertex{0.5}{0.5}
	\vertex{1}{0}
	\vertex{1.5}{0.5}
	\vertex{2}{0}
	\vertex{2}{-0.5}
	\vertex{2.5}{0.25}
	\vertex{2.5}{0.75}
	
	\path[draw] (-2,0) -- (0, 0) -- (1, 0) -- (2, 0);
	\path[draw] (-2, 0.5) -- (-1, 0) -- (-2, -0.5);
	\path[draw] (0, 0) -- (0.5, 0.5) -- (1, 0) -- (2, -0.5);
	\path[draw] (-0.5, 0.5) -- (1.5, 0.5) -- (2.5, 0.75);
	\path[draw] (1.5, 0.5) -- (2.5, 0.25);
}

%% file: 1circuits.tex
\doublecover{
	\draw[pattern=north east lines] (-1,-0.2) rectangle (0, 0.8);
	\path[draw] (0,0) -- (1,0);
	\path[draw] (0, 0.6) -- (1, 0.6);
	\draw[pattern=north east lines] (1,-0.2) rectangle (2, 0.2);
	\draw[pattern=north east lines] (1,0.4) rectangle (2, 0.8);
}{
	\draw[pattern=north east lines] (-1,-0.2) rectangle (0, 0.2);
	\path[draw] (0,0) -- (1,0);
	\draw[pattern=north east lines] (1,-0.2) rectangle (2, 0.2);
}
\hspace{2cm}
\doublecover{
	\path[draw] (1, -0.1) .. controls (0, -0.1) and (0, 0.7) .. (1, 0.7);
	\path[draw] (1, 0.1) .. controls (0.2, 0.1) and (0.2, 0.5) .. (1, 0.5);
	\draw[pattern=north east lines] (1,-0.2) rectangle (2, 0.2);
	\draw[pattern=north east lines] (1,0.4) rectangle (2, 0.8);
}{
	\path[draw] (1, -0.1) .. controls (0, -0.1) and (0, 0.1) .. (1, 0.1);
	\draw[pattern=north east lines] (1,-0.2) rectangle (2, 0.2);
}

%% file: 2circuits.tex

\doublecover{
	\draw[pattern=north east lines] (-1,-0.2) rectangle (0, 0.8);
	\path[draw] (0,0) -- (1,0);
	\path[draw] (0, 0.6) -- (1, 0.6);
	\draw[pattern=north east lines] (1,-0.2) rectangle (2, 0.2);
	\draw[pattern=north east lines] (1,0.4) rectangle (2, 0.8);
	\path[draw] (2, 0) -- (3, 0);
	\path[draw] (2, 0.6) -- (3, 0.6);
	\draw[pattern=north east lines] (3,-0.2) rectangle (4, 0.8);
}{
	\draw[pattern=north east lines] (-1,-0.2) rectangle (0, 0.2);
	\path[draw] (0,0) -- (1,0);
    \draw (0.5, -0.1) node[anchor = north] {2};
	\draw[pattern=north east lines] (1,-0.2) rectangle (2, 0.2);
	\path[draw] (2,0) -- (3, 0);
    \draw (2.5, -0.1) node[anchor = north] {2};
	\draw[pattern=north east lines] (3,-0.2) rectangle (4, 0.2);
}%
\hspace{2cm}
\doublecover{
	\path[draw] (1, -0.1) .. controls (0, -0.1) and (0, 0.7) .. (1, 0.7);
	\path[draw] (1, 0.1) .. controls (0.2, 0.1) and (0.2, 0.5) .. (1, 0.5);
	\draw[pattern=north east lines] (1,-0.2) rectangle (2, 0.2);
	\draw[pattern=north east lines] (1,0.4) rectangle (2, 0.8);
	\path[draw] (2, -0.1) .. controls (3, -0.1) and (3, 0.7) .. (2, 0.7);
	\path[draw] (2, 0.1) .. controls (2.8, 0.1) and (2.8, 0.5) .. (2, 0.5);
}{
	\path[draw] (1, -0.1) .. controls (0, -0.1) and (0, 0.1) .. (1, 0.1);
    \draw (0.2, 0) node[anchor = east] {1};
	\draw[pattern=north east lines] (1,-0.2) rectangle (2, 0.2);
	\path[draw] (2, -0.1) .. controls (3, -0.1) and (3, 0.1) .. (2, 0.1);
    \draw (2.8, 0) node[anchor = west] {1};
}

\vspace{1cm}

\doublecover{
	\path[draw] (1, -0.1) .. controls (0, -0.1) and (0, 0.7) .. (1, 0.7);
	\path[draw] (1, 0.1) .. controls (0.2, 0.1) and (0.2, 0.5) .. (1, 0.5);
	\draw[pattern=north east lines] (1,-0.2) rectangle (2, 0.2);
	\draw[pattern=north east lines] (1,0.4) rectangle (2, 0.8);
	\path[draw] (2, 0) -- (3, 0);
	\path[draw] (2, 0.6) -- (3, 0.6);
	\draw[pattern=north east lines] (3,-0.2) rectangle (4, 0.8);
    \draw (0.2, 0) node[anchor = east] {\phantom{1}};
}{
	\path[draw] (1, -0.1) .. controls (0, -0.1) and (0, 0.1) .. (1, 0.1);
    \draw (0.2, 0) node[anchor = east] {1};
	\draw[pattern=north east lines] (1,-0.2) rectangle (2, 0.2);
	\path[draw] (2, 0) -- (3, 0);
    \draw (2.5, -0.1) node[anchor = north] {2};
	\draw[pattern=north east lines] (3,-0.2) rectangle (4, 0.2);
}%
\hspace{2cm}
\doublecover{
	\draw[pattern=north east lines] (-1,-0.2) rectangle (0, 0.8);
	\path[draw] (0,-0.1) -- (1,-0.1);
	\path[draw] (0, 0.5) -- (1, 0.5);
	\path[draw] (0,0.1) -- (1,0.1);
	\path[draw] (0, 0.7) -- (1, 0.7);
	\draw[pattern=north east lines] (1,-0.2) rectangle (2, 0.2);
	\draw[pattern=north east lines] (1,0.4) rectangle (2, 0.8);
}{
	\draw[pattern=north east lines] (-1,-0.2) rectangle (0, 0.2);
	\path[draw] (0,-0.1) -- (1,-0.1);
	\path[draw] (0, 0.1) -- (1, 0.1);
    \draw (0.5, -0.2) node[anchor = north] {1};
    \draw (0.5, 0.2) node[anchor = south] {1};
	\draw[pattern=north east lines] (1,-0.2) rectangle (2, 0.2);
}

%% file: circuits.tex
\def\s{1.1}

\begin{tikzcd}[row sep = tiny]
    \mbox{Type I}
    &
    \mbox{Type II}
    \\
	\begin{tikzpicture}[scale = \s]
		\orient{
            \draw[postaction = decorate] (0,0) ellipse (1 and 0.2);
            \draw[postaction = decorate] (0,0.6) ellipse (1 and 0.2);  
        }   

        \node[anchor = west] at (1.1, 0) {$\tilde e^-$};
        \node[anchor = east] at (-1.1, 0) {\phantom{$\tilde e^-$}};
        \node[anchor = west] at (1.1, 0.6) {$\tilde e^+$};
        \node[anchor = east] at (-1.1, 0.6) {\phantom{$\tilde e^+$}};
	\end{tikzpicture}
	\arrow[d]
	&
	\begin{tikzpicture}[scale = \s]
        \orient{
            \path[draw, postaction = decorate] (0, 0) -- (1,0);
            \path[draw, postaction = decorate] (0, 1) -- (1, 1);
    		\spiral{1}{0}
    		\spiralleft{0}{0}
        }

        \node[anchor=east] at (-1.7, 0.5) {$\tilde e_1^+$};
		\node[anchor=north] at (0.5, -0.05) {$\tilde e_2^-$};
        \node[anchor=west] at (2.7, 0.5) {$\tilde e_3^+$};

        \node[anchor=east] at (-0.8, 0.5) {$\tilde e_1^-$};
		\node[anchor=south] at (0.5, 1.05) {$\tilde e_2^+$};
        \node[anchor=west] at (1.8, 0.5) {$\tilde e_3^-$};
	\end{tikzpicture}
	\arrow[d]
    \\
	\begin{tikzpicture}[scale = \s]
        \orient{
            \draw[postaction = decorate] (0,0) ellipse (1 and 0.2);
        } 
		
		\node[anchor = west] at (1.1, 0) {$e$};
        \node[anchor = east] at (-1.1, 0) {\phantom{$e$}};
	\end{tikzpicture}
	&
	\begin{tikzpicture}[scale = \s]

        \orient{
            \drawselfloopleft{0}{0}
    		\path[draw, postaction = decorate] (0,0) -- (1,0);
    		\drawselfloop{1}{0}
        }
  
		\node[anchor=east] at (-0.8, -0.45) {$e_1$};
		\node[anchor=north] at (0.5, -0.05) {$e_2$};
        \node[anchor=west] at (1.8, -0.45) {$e_3$};

	\end{tikzpicture}
    \\
    \gamma_C=\tilde e^+ - \tilde e^-
    &
    \gamma_C = -(\tilde e_1^+ - \tilde e_1^-) + 2(\tilde e_2^+ - \tilde e_2^-) + (\tilde e_3^+ - \tilde e_3^-)
    \\[0.8cm]
    %
    %
    \mbox{Type III}
    &
    \mbox{Type IV}
    \\
    \begin{tikzpicture}[scale = \s]
		\orient{
            \spiral{0}{0}
    		\spiralleft{0}{0}
        }   

        \node[anchor=east] at (-1.7, 0.5) {$\tilde e_1^+$};
        \node[anchor=west] at (1.7, 0.5) {$\tilde e_2^+$};

        \node[anchor=east] at (-0.8, 0.5) {$\tilde e_1^-$};
        \node[anchor=west] at (0.8, 0.5) {$\tilde e_2^-$};
	\end{tikzpicture}
	\arrow[d]
    &
    \begin{tikzpicture}[scale = \s]

		\vertex[2]{-0.5}{0.5}

        \orient{
            \path[draw, postaction = decorate] (-0.5, 0.5) -- (1,0);
    		\path[draw, postaction = decorate] (-0.5, 0.5) -- (1,1);
    		\spiral{1}{0}        
        }

        \node[anchor=south] at (0.5, 0.9) {$\tilde e_1^+$};
        \node[anchor=north] at (0.5, 0.1) {$\tilde e_1^-$};
        
        \node[anchor=west] at (2.7, 0.5) {$\tilde e_2^+$};
        \node[anchor=west] at (1.8, 0.5) {$\tilde e_2^-$};
	\end{tikzpicture}
	\arrow[d]
    \\
	\begin{tikzpicture}[scale = \s]
		\orient{
            \drawselfloopleft{0}{0}
    		\drawselfloop{0}{0}
        }   

        \node[anchor=east] at (-0.8, -0.45) {$e_1$};
		\node[anchor=west] at (0.8, -0.45) {$e_2$};
	\end{tikzpicture}
    &
    \begin{tikzpicture}[scale = \s]
		\vertex{-0.5}{0}

        \orient{
            \path[draw, postaction = decorate] (-0.5,0) -- (1,0);
            \drawselfloop{1}{0}
        }

        \node[anchor=north] at (0.5, -0.05) {$e_1$};
		\node[anchor=west] at (1.8, -0.45) {$e_2$};
        \node[anchor=west] at (2.7, 0) {\phantom{$\tilde e_2^-$}};
	\end{tikzpicture}
	\\
    \gamma_C = (\tilde e_1^+ - \tilde e_1^-) - (\tilde e_2^+ - \tilde e_2^-)
    &
    \gamma_C = 2(\tilde e_1^+ - \tilde e_1^-) + (\tilde e_2^+ - \tilde e_2^-)
    \\[0.8cm]
	%
    %
    \mbox{Type V}
    &
    \mbox{Type VI}
    \\
	\begin{tikzpicture}[scale = \s]
		\clip (-2*\vertexsize, -0.5) rectangle (2.1, 0.6);
		\vertex[2]{0}{0}

        \orient{
            \path[draw, postaction = decorate] (0, 0) .. controls (2, -1.2) and (2, 0.3) .. (0,0);
    		\path[draw, postaction = decorate] (0, 0) .. controls (2, -0.3) and (2, 1.2) .. (0,0);
        }

        \node[anchor = west] at (1.6, -0.3) {$\tilde e^-$};
        \node[anchor = west] at (1.6, 0.3) {$\tilde e^+$};
	\end{tikzpicture}
	\arrow[d]
	&
	\begin{tikzpicture}[scale = \s]
		\vertex[2]{0}{0} 
		\vertex[2]{2}{0}

        \orient{
            \path[draw, postaction = decorate] (0, 0) to[bend left = 30] (2,0);
            \path[draw, postaction = decorate] (0, 0) to[bend right = 30] (2,0);
        }

        \node[anchor = west] at (0.5, -0.5) {$\tilde e^-$};
        \node[anchor = west] at (0.5, 0.5) {$\tilde e^+$};
	\end{tikzpicture}
	\arrow[d]
	\\
	\begin{tikzpicture}[scale = \s]
		\clip (-2*\vertexsize, -0.5) rectangle (1.9, 0.5);
		\vertex{0}{0}
		\node[anchor=west] at (1.6, 0) {$e$};

        \orient{
            \drawselfloop{0}{0}
        }		
	\end{tikzpicture}	
	&
	\begin{tikzpicture}[scale = \s]
		\vertex{0}{0}
		\vertex{2}{0}
		
        \orient{
            \path[draw, postaction = decorate] (0,0) -- (2,0);
        }
        
		\node[anchor=north] at (0.5, -0.05) {$e$};
	\end{tikzpicture}
    \\
    \gamma_C = \tilde e^+ - \tilde e^-
    &
    \gamma_C = \tilde e^+ - \tilde e^-
\end{tikzcd}